\documentclass{amsart}
\usepackage{amsmath}
\usepackage{amssymb}
\usepackage{amscd}
\usepackage{multicol}
\usepackage{MnSymbol}
\usepackage{tikz}
\usepackage[OT2,T1]{fontenc}
\usepackage{mathrsfs}
\usepackage{comment} 
\usepackage{stmaryrd}
\usepackage[all]{xy}
\usepackage{longtable}
\usepackage{multirow,bigdelim}
\DeclareSymbolFont{cyrletters}{OT2}{wncyr}{m}{n}
\DeclareMathSymbol{\Sha}{\mathalpha}{cyrletters}{"58}

\title{Alternating multizeta values in positive characteristic}
\author{Ryotaro Harada}
\address{Graduate School of Mathematics, Nagoya University, 
Furo-cho, Chikusa-ku, Nagoya 464-8602 Japan }
\email{m15039r@math.nagoya-u.ac.jp}
\date{September 9, 2019.}

\newtheorem{thm}{Theorem}[section]
\newtheorem{lem}[thm]{Lemma}
\newtheorem{cor}[thm]{Corollary}
\newtheorem{prop}[thm]{Proposition}
\theoremstyle{remark}
        
\theoremstyle{definition}
\newtheorem{defn}[thm]{Definition}
\newtheorem{rem}[thm]{Remark}
\newtheorem{nota}[thm]{Notation}     

\newtheorem{eg}[thm]{Example}       
\newtheorem{conj}[thm]{Conjecture}    
\newtheorem{q}[thm]{Question}

\begin{document}
\bibliographystyle{amsalpha+}

\begin{abstract}
We introduce alternating multizeta values in positive characteristic
which are generalizations of Thakur multizeta values.
We establish their fundamental properties including non-vanishing,
sum-shuffle relations, period interpretation and linear independence
which is a direct  sum result for these values.
\end{abstract}


\maketitle
\tableofcontents
\setcounter{section}{-1}
\section{Introduction}
In this paper, we introduce the {\it alternating multizeta vaues in positive characteristic} (AMZVs in short) which are defined as the following infinite sums. For $\mathfrak{s}=(s_1, \ldots, s_r)\in\mathbb{N}^r$ and ${\boldsymbol \epsilon}=(\epsilon_1, \ldots, \epsilon_r) \in{(\mathbb{F}_q^{\times})}^r$,
\begin{align}\label{amz1}
    \zeta_A(\mathfrak{s};\boldsymbol{\epsilon})=\sum_{\substack{a_1, \ldots, a_r\in A_+\\ \deg a_1>\cdots>\deg a_r\geq 0}}\frac{\epsilon_1^{\deg a_1} \cdots \epsilon_r^{\deg a_r}}{a_1^{s_1} \cdots a_r^{s_r}}\in k_{\infty}.
\end{align}
For the definitions of $\mathbb{F}_q$, $A$, $A_+$, $k$, $k_{\infty}$ and $\overline{k}$, see $\S\ref{No}$.
We call ${\rm wt}(\mathfrak{s}):=\sum_{i=1}^{r}s_i$ the weight and ${\rm dep}(\mathfrak{s}):=r$ the depth of the presentation of $\zeta_A(\mathfrak{s};\boldsymbol{\epsilon})$. We also study their properties listed in below.
\begin{itemize}
\item[(A).] Non-vanishing (Theorem\ref{nonvan}) 
\item[(B).] Sum-shuffle relations (Theorem \ref{shamz})
\item[(C).] Period interpretation (Theorem \ref{perint})
\item[(D).] Linear independence (Theorem \ref{linindamz})
\end{itemize}
We also note that $\zeta_A(\mathfrak{s};\boldsymbol{\epsilon})$ is generalization of Thakur multizeta values, positive characteristic analogue of multizeta values. It is said that multizeta values originated from the research in Euler's work \cite{Eu} in 1776. He introduced them (originally, $r=2$ case) as the following infinite series:
\[
	\zeta(\mathfrak{s}):=\sum_{n_1>\cdots >n_r}\frac{1}{n_1^{s_1}\cdots n_r^{s_r}}\in\mathbb{R}
\]
where $\mathfrak{s}:=(s_1, \ldots, s_r)\in\mathbb{N}^r$ and $s_1>1$. Moreover, $r$ is called depth and $\sum_{i=1}^rs_i$ is called weight of the presentation of $\zeta(s_1, \ldots, s_r)$.  In the last quarter century, it got known that they have connection to number theory \cite{DG, Z}, knot theory \cite{LM} and quantum field theory \cite{BK} and so on. It is also known that there are linear/algebraic relations between multizeta values. In particular, they satisfy two types algebraic relations, namely the integral-shuffle and sum-shuffle relations (also known as harmonic relation, stuffle relation) each obtained from integral expression and series expression respectively. For the lower depth case, they are described as following:
\begin{itemize}
\item[] Integral-shuffle relation: 
\begin{align*} 
  \zeta(s_1)\zeta(s_2)=\sum^{s_2-1}_{m=0}\binom{s_1+m-1}{k}\zeta(s_1+m, s_2-m)+\sum^{s_1-1}_{n=0}\binom{s_2+n-1}{n}\zeta(s_2+n, s_1-n),
\end{align*}  
\item[] Sum-shuffle relation: 
\begin{align*}  
  \zeta(s_1)\zeta(s_2)=\zeta(s_1,s_2)+\zeta(s_2,s_1)+\zeta(s_1+s_2).
\end{align*}
\end{itemize}
From these algebraic relations, we may consider an algebra generated by 1 and all multizeta values.
On the other hand, there still exist some open problems. 
For example, we have the following conjecture (folklore):
\begin{conj}\label{gconj}
Let $\overline{\mathfrak{Z}}$ be the $\overline{\mathbb{Q}}$-algebra generated by multizeta values and $\overline{\mathfrak{Z}}_w$ be the $\overline{\mathbb{Q}}$-vector space spanned by the weight $w$ multizeta values for $w\geq 2$. Then one has the following: 
\begin{itemize}
\item[(i)] $\overline{\mathfrak{Z}}$ forms an weight-graded algebra, that is, $\overline{\mathfrak{Z}}=\overline{\mathbb{Q}}\bigoplus_{w\geq 2}\overline{\mathfrak{Z}}_w$;
\item[(ii)] $\overline{\mathfrak{Z}}$ is defined over $\mathbb{Q}$, i.e. the canonical map $\overline{\mathbb{Q}}\otimes_{\mathbb{Q}}\mathfrak{Z}\rightarrow\overline{\mathfrak{Z}}$ is bijective.
\end{itemize}
\end{conj}
We note that this conjecture is a stronger version of Goncharov conjecture in \cite{Go}. 
Because in the original conjecture, $\mathbb{Q}$ is a base field instead of $\overline{\mathbb{Q}}$ and thus it follows from Conjecture \ref{gconj}.
Moreover, 
several variants of multizeta values were also invented such as \cite{KZ}, \cite{Fu} and \cite{Br12} etc.
There exist the alternating multizeta values\footnote{They are also known as Euler sums \cite{Br0} or level 2 colored multizeta values \cite{BJOP}.} which are certain generalizations of multizeta values and defined as following series:
\[
    \zeta(\mathfrak{s}; \boldsymbol{\epsilon}):=\sum_{n_1>\cdots>n_r>0}\frac{\epsilon_1^{n_1} \cdots \epsilon_r^{n_r}}{n_1^{s_1} \cdots n_r^{s_r}}\in\mathbb{R}
\]
where $\mathfrak{s}:=(s_1, \ldots, s_r)\in\mathbb{N}^r$, $\boldsymbol{\epsilon}:=(\epsilon_1, \ldots, \epsilon_r) \in\{\pm 1\}^r$ and $(s_1, \epsilon_1)\neq(1, 1)$.
In these two decades, they were also studied by Broadhurst, Deligne-Goncharov, Glanois, Hoffman, Kaneko-Tsumura and Zhao. Due to their contributions, now it is known that alternating multizeta values have connection to study of knots \cite{Br0}, Feynman diagrams \cite{Br}, modular forms \cite{KT}, mixed Tate motives over $\mathbb{Z}[1/2]$ \cite{DG}. It is also shown that alternating multizeta values enjoy integral-shuffle and sum-shuffle relation. 
There also exist finite variant \cite{BTT}, $p$-adic variant \cite{Un}, \cite{Jaro}, and motivic variant \cite{Gl}. 

In characteristic $p$ case, an analogue of multizeta values were invented by Thakur \cite{Th}.
First we recall the power sums. For $s \in \mathbb{Z}$ and $d\in \mathbb{Z}_{\geq0}$, power sums are defined by
\[
  S_d(s):=\sum_{ a\in A_{d+}}\frac{1}{a^s}\in k.
\] 
For $\mathfrak{s}=(s_1, \ldots, s_r)\in\mathbb{N}^r$ and $d\in\mathbb{Z}_{\geq0}$, we define 
\[
  S_d(\mathfrak{s}) := S_d(s_1)\sum_{d>d_2>\cdots >d_r\geq0}S_{d_2}(s_2)\cdots S_{d_r}(s_r)\in k.
\]
For $s \in \mathbb{Z}_{\geq 0}$, the {\it Carlitz zeta values}  are defined by
\[
  \zeta_A(s):=\sum_{a\in A_{+}}\frac{1}{a^s}\in k_{\infty}.
\]
Thakur generalized this definition to that of {\it multizeta values in positive characteristic} (MZVs in short). For $\mathfrak{s}=(s_1, \ldots, s_r)\in\mathbb{N}^r$, 
\begin{align*}
  \zeta_A(\mathfrak{s}):= \sum_{d_1>\cdots >d_r\geq0}S_{d_1}(s_1)\cdots S_{d_r}(s_r)
                                   =\sum_{\substack{a_1, \ldots, a_r\in A_+\\ \deg a_1>\cdots>\deg a_r\geq 0}}\frac{1}{a^{s_1}_1\cdots a^{s_r}_r}\in k_{\infty}.
\end{align*}

He and Anderson developed the following properties for MZVs:    
\begin{itemize}
\item[(a).] Non-vanishing (\cite[Theorem 4]{T}) 
\item[(b).] Sum-shuffle relations (\cite[Theorem 3]{T1})
\item[(c).] Period interpretation (\cite[Theorem 1]{AT1})
\end{itemize}
About (b), Thakur showed that the product of two MZVs can be expressed as an $\mathbb{F}_p$-linear combination of some MZVs of the same weight and we follow Thakur's terminology to call these sum-shuffle relations. Thus the MZVs form an $\mathbb{F}_p$-algebra.
By lacking of the integral expression, we do not know the existence of the integral-shuffle relation in characteristic $p$. 
By using the transcendence theory developed by Yu \cite{JY}, Anderson-Brownawell-Papanikolas \cite{ABP} and Papanikolas \cite{Pa} together with period interpretations of MZVs by Anderson-Thakur \cite{AT1}, some advances on MZVs are established by Chang-Yu \cite{CY07} (for Carlitz zeta values), Chang \cite{C1, C2}, Chang-Papanikolas-Yu \cite{CPY} and Mishiba \cite{Y1, Y2}. 
For example, in \cite{C1}, a characteristic $p$ analogue of Conjecture \ref{gconj} was proved as the following property:
\begin{itemize}
\item[(d).] Linear independence (\cite[Theorem 2.2.1]{C1})
\end{itemize}

We emphasize that AMZVs have important and fundamental properties which generalize those of MZVs.
Indeed, we show them consisting of non-vanishing (Theorem \ref{nonvan}), sum-shuffle relations (Theorem \ref{shamz}) and period interpretation (Theorem \ref{perint}) as an alternating analogue of (a), (b), and (c). 
By applying those properties, we will show the alternating analogue of (d) (Corollary \ref{gonamz}) in \S\ref{seclinind} and thereby obtain the transcendence of each alternating multizeta value. These result give us the generalization of main theorems in \cite{C1, T, T1}. 

From these results, while we find differences in observing alternating analogue of (b) (cf. \S\ref{subsecsh}), (c) (cf. \S\ref{secperint}) and Corollary \ref{gonamz} (cf. \S\ref{seclinind}) by units $\gamma\in\overline{\mathbb{F}}_q^{\times}$ and $\epsilon\in\mathbb{F}_q^{\times}$, their basic properties appeared in this paper are similar to those of MZVs. More precisely, for the property (A), it is immediately obtained by an inequality property of the absolute values of power sums proved by Thakur \cite{T}. For the property (B), we use Chen's formula \cite{HJC} and approach the higher depth case by induction method invented by Thakur \cite{T1}. For the property (C), inspired by Anderson-Thakur polynomials (\cite{AT}) that can interpolate power sums, we use those polynomials to create suitable power series that their specialization are AMZVs and then we use these series to create suitable pre-$t$-motives to establish the period interpretation of (C). This result is also indicated from Chang's advice on depth 1 AMZVs. For the linear independence result of AMZVs, we use Chang's method \cite{C1} by applying Anderson-Brownawwell-Papanikolas criterion \cite{ABP} to establish the alternating analogue of MZ property for AMZVs. By the linear independence result, we have (D), that is, an alternating analogue of (d) (for the detail, see Theorem \ref{gonamz}):
\begin{itemize}
\item[(i)] $\overline{\mathcal{AZ}}$ forms an weight-graded algebra, that is, $\overline{\mathcal{AZ}}=\overline{k}\bigoplus_{w\in\mathbb{N}}\overline{\mathcal{AZ}}_w$,
\item[(ii)] $\overline{\mathcal{AZ}}$ is defined over $k$, that is, we have the canonical map $\overline{k}\otimes_{k}\mathcal{AZ}\rightarrow\overline{\mathcal{AZ}}$ which is bijective.  
\end{itemize}

\section{Notations and Definitions}
\label{No}
We put the following notations.

\begin{itemize}
\setlength{\leftskip}{1.0cm}
\item[$q$] \quad  a power of a prime number $p$.  
\item[$\mathbb{F}_q$] \quad a finite field with $q$ elements.
\item[$\theta$, $t$] \quad independent variables.
\item[$A$] \quad the polynomial ring $\mathbb{F}_q[\theta]$.
\item[$A_{+}$] \quad the set of monic polynomials in $A$.
\item[$A_{d+}$] \quad the set of elements in $A_{+}$ of degree $d$. 
\item[$k$] \quad the rational function field $\mathbb{F}_q(\theta)$.
\item[$k_{\infty}$] \quad the completion of $k$ at the infinite place $\infty$, $\mathbb{F}_q((\frac{1}{\theta}))$.
\item[$\overline{k_{\infty}}$] \quad a fixed algebraic closure of $k_{\infty}$.
\item[$\mathbb{C}_{\infty}$] \quad the completion of $\overline{k_{\infty}}$ at the infinite place $\infty$.
\item[$\overline{k}$] \quad a fixed algebraic closure of $k$ in $\mathbb{C}_{\infty}$.
\item[$|\cdot|_{\infty}$]\quad a fixed absolute value for the completed field $\mathbb{C}_{\infty}$ so that $|\theta|_{\infty}=q$.
\item[$\mathbb{T}$]\quad the Tate algebra over $\mathbb{C}_{\infty}$, the subring of $\mathbb{C}_{\infty}[[t]]$ consisting of\\
\quad power series convergent on the closed unit disc $|t|_{\infty}\leq 1$. 
\item[$D_i$]\quad $\prod^{i-1}_{j=0}(\theta^{q^i}-\theta^{q^j})\in A_{+}$ where $D_0:=1$.
\item[$\Gamma_{n+1}$]\quad the Carlitz gamma, $\prod_{i}D_i^{n_i}$ ($n = \sum_{i}n_iq^i\in\mathbb{Z}_{\geq0} \ (0\leq n_i\leq q-1)$).  
\end{itemize}

In the function field case, $A$ is an analogue of the integer ring $\mathbb{Z}$. Thus we may consider  $A^{\times}=\mathbb{F}_q^{\times}$ as an analogue of $\mathbb{Z}^{\times}=\{\pm 1\}$. 

We define alternating power sums and alternating multizeta values in positive characteristic along the construction of MZVs by Thakur.
For $s\in\mathbb{N}$, $\epsilon \in\mathbb{F}_q^{\times}$ and $d\in\mathbb{Z}_{\geq 0}$, we define the {\it alternating power sums} by
\[
    S_d(s;\epsilon):=\epsilon^{d}S_d(s)=\sum_{a\in A_{d+}}\frac{\epsilon^d}{a^s}\in k.
\]
The above alternating power sums are extended inductively as follows.

For $\mathfrak{s}=(s_1, \ldots, s_r)\in\mathbb{N}^r$, ${\boldsymbol \epsilon}=(\epsilon_1, \ldots, \epsilon_r) \in{(\mathbb{F}_q^{\times})}^r$ and $d\in\mathbb{Z}_{\geq 0}$, we define 
\[
	S_{<d}({\mathfrak s};{\boldsymbol \epsilon}):=\sum_{d>d_1>\cdots >d_r\geq 0}S_{d_1}(s_1;\epsilon_1)\cdots S_{d_r}(s_r;\epsilon_r)\in k
\]
and
\begin{align}\label{extpowsum}
    S_d(\mathfrak{s};\boldsymbol{\epsilon})&:=S_d(s_1;\epsilon_1)S_{<d}(s_2, \ldots, s_r; \epsilon_2, \ldots, \epsilon_r)\\
    &:=S_d(s_1;\epsilon_1)\sum_{d>d_2>\cdots>d_r\geq 0}S_{d_2}(s_2;\epsilon_2)\cdots S_{d_r}(s_r;\epsilon_r)\in k.\nonumber
\end{align}
When $r-1>d$, $S_d(\mathfrak{s};\boldsymbol{\epsilon})=0$ since it is empty sum.
By using these alternating power sums, AMZVs (cf. \eqref{amz1} ) are interpreted as follows. 
\begin{align}\label{chpes}
    \zeta_A(\mathfrak{s};\boldsymbol{\epsilon})&=\sum_{d\geq0}S_d(\mathfrak{s};\boldsymbol{\epsilon})\in k_{\infty}.
\end{align}
\begin{rem}
We remark that $\zeta_A(\mathfrak{s};\boldsymbol{\epsilon})$ specializes to $\zeta_A(\mathfrak{s})$ when ${\boldsymbol \epsilon}=(\epsilon_1, \ldots, \epsilon_r)=(1, \ldots, 1)$.
\end{rem}
\begin{rem}
In writing this paper, the author got to know that there exist the colored variant of MZVs by communicating with Thakur.
They were defined by his students Qibin Shen and Shuhui Shi as follows. 
For $\mathfrak{s}=(s_1, \ldots, s_r)\in\mathbb{N}^r$, and $n\in \mathbb{N}$, let ${\boldsymbol \xi}=(\xi_1, \ldots, \xi_r)\in(\overline{\mathbb{F}_q}^{\times})^r$ so that each $\xi_i$ is $n$-th root of unity. Then level $n$ colored MZVs are the following series:
\begin{align}\label{cmzv}
    \zeta_A(\mathfrak{s};{\boldsymbol \xi})&:=\sum_{\substack{a_1, a_2, \ldots, a_r\in A_+\\ \deg a_1>\deg a_2>\cdots>\deg a_r\geq 0}}\frac{\xi_1^{\deg a_1}\xi_2^{\deg a_2} \cdots \xi_r^{\deg a_r}}{a_1^{s_1}a_2^{s_2} \cdots a_r^{s_r}}.
\end{align}
This $\zeta_A(\mathfrak{s};{\boldsymbol \xi})$ includes AMZVs as level $q-1$ case and we remark that the $q-1$ case was defined by the author and Shen-Shi independently. 
\end{rem}
In this paper, we write $n$-fold Frobenius twisting as follows.
\begin{align*}
	\mathbb{C}_{\infty}((t))&\rightarrow\mathbb{C}_{\infty}((t))\\
	f:=\sum_{i}a_it^i&\mapsto \sum_{i}a_i^{q^{n}}t^i=:f^{(n)}.
\end{align*} 
Moreover, we fix a fundamental period $\tilde{\pi}$ of the Carlitz module (see \cite{Goss, Th}). We define the following power series.
\[
	\Omega=\Omega(t):=(-\theta)^{-q/(q-1)}\prod_{i=1}^{\infty}(1-t/\theta^{q^i})\in\mathbb{C}_{\infty}[[t]]
\]
where $(-\theta)^{1/(q-1)}$ is a fixed $(q-1)$st root of $-\theta$ so that $\frac{1}{\Omega(\theta)}=\tilde{\pi}$ (\cite{ABP, AT1}). 
Here, we also introduce another expression of $\zeta_A(\mathfrak{s};\boldsymbol{\epsilon})$ by using the following theorem which was shown by Anderson and Thakur. 

\begin{thm}[\cite{AT}]
For each $s\in\mathbb{N}$, there exists an unique polynomial $H_s=H_s(t)\in A[t]$ such that 
\begin{align}\label{atint}
    (H_{s-1}\Omega^s)^{(d)}|_{t=\theta}=\Gamma_sS_d(s)/\tilde{\pi}^{s}
\end{align}
for all $d\in\mathbb{Z}_{\geq 0}$ and $s\in \mathbb{N}$. Moreover, when we regard $H_s$ as a polynomial of $\theta$ over $\mathbb{F}_q[t]$ by $A[t]=\mathbb{F}_q[t][\theta]$ then
\begin{align}\label{degh}
    \deg_{\theta}H_s\leq \frac{sq}{q-1}.
\end{align}
\end{thm}
This polynomial $H_s$ is called the {\it Anderson-Thakur polynomial}. 
From \eqref{chpes} and \eqref{atint}, we obtain the following expression of $\zeta_A(\mathfrak{s};\boldsymbol{\epsilon})$;
\begin{align}
 \zeta_A(\mathfrak{s};\boldsymbol{\epsilon})=\frac{\tilde{\pi}^{s_1+\cdots s_r}}{\Gamma_{s_1}\cdots\Gamma_{s_r}} \sum_{d_1>\cdots>d_r\geq 0}\epsilon_1^{d_1}(H_{s_1-1}\Omega^{s_1})^{(d_1)}|_{t=\theta}\cdots\epsilon_r^{d_r}(H_{s_r-1}\Omega^{s_r})^{(d_r)}|_{t=\theta}.
\end{align}

\section{Fundamental properties}
In this section, we prove the non-vanishing and sum-shuffle relation of AMZVs.

\subsection{Non-vanishing property of AMZVs} 
We show the non-vanishing property as the following theorem by using valuation of power sums which evaluated by Thakur \cite{T}.  
\begin{thm}\label{nonvan}
For any ${\mathfrak s}=(s_1, \ldots, s_r)\in\mathbb{N}^r$ and ${\boldsymbol \epsilon}=(\epsilon_1, \ldots, \epsilon_r) \in{(\mathbb{F}_q^{\times})}^r$, $\zeta_A(\mathfrak{s};\boldsymbol{\epsilon})$ are non-vanishing.
\end{thm}
\begin{proof}
From \eqref{chpes}, we can write $\zeta_A(\mathfrak{s};\boldsymbol{\epsilon})$ as follows.
\begin{align*}
	 \zeta_A(\mathfrak{s};\boldsymbol{\epsilon})=&\sum_{d_1>d_2>\cdots >d_r\geq0}\epsilon_1^{d_1}\epsilon_2^{d_2}\cdots\epsilon_r^{d_r}S_{d_1}(s_1)S_{d_2}(s_2)\cdots S_{d_r}(s_r)
\end{align*}
On the other hand, in \cite{T}, Thakur showed that
\[
	\deg_{\theta}S_d(k)>\deg_{\theta}S_{d+1}(k).
\]
Therefore we have
\begin{align*}
	|\zeta_A(s_1, \ldots, s_r; \epsilon_1, \ldots, \epsilon_r)|_{\infty}&=|\sum_{d_1>d_2>\cdots >d_r\geq0}\epsilon_1^{d_1}\epsilon_2^{d_2}\cdots\epsilon_r^{d_r}S_{d_1}(s_1)S_{d_2}(s_2)\cdots S_{d_r}(s_r)|_{\infty}\\
	&=|S_{r-1}(s_1)S_{r-2}(s_2)\cdots S_{0}(s_r)|_{\infty}\\
	&\neq0
\end{align*}
by using $\deg_{\theta}S_0(k)=0$ and $\deg_{\theta}S_d(k)<0$ $( k>0, d>0 )$ in \cite[\S 2.2.3.]{T}. 
Thus $\zeta_A({\mathfrak s}; {\boldsymbol \epsilon})$ are non-vanishing.
\end{proof}

\subsection{Sum-shuffle relation for AMZVs}\label{subsecsh}
In this section, we give sum-shuffle relations for our $\zeta_A({\mathfrak s}; {\boldsymbol \epsilon})$. This kind of relations show that products of two  AMZVs are expressed by $\mathbb{F}_{p}$-linear combination of AMZVs with preserving their weights.
From the relation, $\zeta_A({\mathfrak s}; {\boldsymbol \epsilon})$ form an $\mathbb{F}_{p}$-algebra. 

For the products of power sums $S_d(s)$, the following formula was shown by Chen \cite{HJC}.
\begin{prop}[\cite{HJC} Theorem 3.1]\label{chen}
For $s_1, s_2\in\mathbb{N}$, we have
\begin{align*}
	S_d(s_1)&S_d(s_2)-S_d(s_1+s_2)=\sum_{\substack{0<j<s_1+s_2\\q-1|j}}\Delta^{j}_{s_1, s_2}S_d(s_1+s_2-j, j)
\end{align*}
where 
\begin{align}\label{deltadef}
\Delta^{j}_{s_1, s_2}=(-1)^{s_1-1}\binom{j-1}{s_1-1}+(-1)^{s_2-1}\binom{j-1}{s_2-1}.
\end{align}
\end{prop}
The key idea to prove Proposition \ref{chen} is the following partial fraction decomposition
\begin{align*} 
\frac{1}{a^{s_1}b^{s_2}}=\sum_{0<j<s_1+s_2}\Biggl\{ \frac{(-1)^{s_1-1}\binom{j-1}{s_1-1}}{a^{s_1+s_2-j} (a-b)^j }+\frac{(-1)^{s_2-1}\binom{j-1}{s_2-1}}{b^{s_1+s_2-j} (a-b)^{j} } \Biggr\}
\end{align*}
for $a, b\in A\backslash\{0\}$.
By using the above decomposition, she obtained the following formula in the proof of \cite[Theorem3.1]{HJC}:
\begin{align}\label{pfdt}
\sum_{\substack{ a\neq b\in A_{+}\\ \deg a=\deg b}}\frac{1}{a^{s_1}b^{s_2}}=\sum_{\substack{0<j<s_1+s_2\\(q-1)|j}}\Biggl\{\sum_{\substack{ a,b\in A_{+}\\ \deg a>\deg b}} \frac{(-1)^{s_1-1}\binom{j-1}{s_1-1}}{a^{s_1+s_2-j}b^j  }+\sum_{ \substack{ a,b\in A_{+}\\ \deg b>\deg a} }\frac{(-1)^{s_2-1}\binom{j-1}{s_2-1}}{b^{s_1+s_2-j}a^{j} } \Biggr\}.
\end{align}
Chen's method is applied to prove the following lemma.
\begin{lem}\label{hl1}
For $s_1, s_2\in\mathbb{N}$, $\epsilon_1, \epsilon_2\in\mathbb{F}_q^{\times}$ and $d\in\mathbb{Z}_{\geq 0}$, we have
\begin{align}\label{lempowsh}
	S_d(s_1;\epsilon_1)&S_d(s_2;\epsilon_2)-S_d(s_1+s_2;\epsilon_1\epsilon_2)=\sum_{\substack{0<j<s_1+s_2\\q-1|j}}\Delta^{j}_{s_1, s_2}S_d(s_1+s_2-j, j; \epsilon_1\epsilon_2, 1)
\end{align}
\end{lem}
\begin{proof}
We give a proof based on the idea in \cite{HJC}.
\begin{align*}
S_d(s_1;\epsilon_1)&S_d(s_2;\epsilon_2)-S_d(s_1+s_2;\epsilon_1\epsilon_2)\\
&=\sum_{a\in A_{d+}}\frac{\epsilon_1^{d}}{a^{s_1}}\sum_{b\in A_{d+}}\frac{\epsilon_2^{d}}{b^{s_2}} -S_d(s_1+s_2;\epsilon_1\epsilon_2)\\
&=\sum_{\substack{a,b\in A_{\plus}\\ d=\deg a>\deg b\geq0}}\frac{(\epsilon_1\epsilon_2)^d}{a^{s_1}b^{s_2}}+\sum_{\substack{a,b\in A_{\plus}\\ d=\deg b>\deg a\geq0}}\frac{(\epsilon_1\epsilon_2)^d}{b^{s_2}a^{s_1}}+\sum_{\substack{a=b\in A_{d\plus}\\ d=\deg a=\deg b}}\frac{(\epsilon_1\epsilon_2)^d}{a^{s_1}b^{s_2}}\\
&\quad +\sum_{\substack{a\neq b\in A_{\plus}\\ d=\deg a=\deg b}}\frac{(\epsilon_1\epsilon_2)^d}{a^{s_1}b^{s_2}} -S_d(s_1+s_2;\epsilon_1\epsilon_2).\\
\intertext{By \eqref{extpowsum}, we have $\sum_{\substack{a=b\in A_{\plus}\\ d=\deg a=\deg b}}\frac{(\epsilon_1\epsilon_2)^d}{a^{s_1}b^{s_2}}=S_d(s_1+s_2;\epsilon_1\epsilon_2)$ and thus } 
\sum_{\substack{a,b\in A_{\plus}\\ d=\deg a>\deg b\geq0}}&\frac{(\epsilon_1\epsilon_2)^d}{a^{s_1}b^{s_2}}+\sum_{\substack{a,b\in A_{\plus}\\ d=\deg b>\deg a\geq0}}\frac{(\epsilon_1\epsilon_2)^d}{b^{s_2}a^{s_1}}+\sum_{\substack{a=b\in A_{d\plus}\\ d=\deg a=\deg b}}\frac{(\epsilon_1\epsilon_2)^d}{a^{s_1}b^{s_2}}\\
&\qquad\quad \ +\sum_{\substack{a\neq b\in A_{\plus}\\ d=\deg a=\deg b}}\frac{(\epsilon_1\epsilon_2)^d}{a^{s_1}b^{s_2}} -S_d(s_1+s_2;\epsilon_1\epsilon_2)\\
&=S_d(s_1, s_2; \epsilon_1\epsilon_2, 1)+S_d(s_2, s_1;\epsilon_1\epsilon_2, 1)+\sum_{\substack{a\neq b\in A_{\plus} \\ d=\deg a=\deg b\geq 0}}\frac{(\epsilon_1\epsilon_2)^d}{a^{s_1}b^{s_2}}.
\end{align*}
By \eqref{pfdt}, we have
\begin{align*}
\sum_{\substack{ a\neq b\in A_{\plus}\\d=\deg a=\deg b\geq 0}}\frac{(\epsilon_1\epsilon_2)^d}{a^{s_1}b^{s_2}}&=\sum_{\substack{0<j<s_1+s_2\\q-1|j}}\Biggl\{\sum_{\substack{ x, y\in A_{+}\\ d=\deg x>\deg y}} \frac{(-1)^{s_1-1}\binom{j-1}{s_1-1} (\epsilon_1\epsilon_2)^d }{x^{s_1+s_2-j}y^j  }\\
&\qquad\qquad\qquad +\sum_{ \substack{ x, y\in A_{+}\\ d=\deg y>\deg x} }\frac{(-1)^{s_2-1}\binom{j-1}{s_2-1}(\epsilon_1\epsilon_2)^d}{y^{s_1+s_2-j}x^{j} } \Biggr\}\\
&=\sum_{\substack{0<j<s_1+s_2\\q-1|j}}\Delta^{j}_{s_1, s_2}S_d(s_1+s_2-j, j; \epsilon_1\epsilon_2, 1).
\end{align*}

Therefore the relation \eqref{lempowsh} is proved.
\end{proof}
\begin{rem}
In Lemma \ref{hl1}, coefficients $\Delta^j_{s_1, s_2}$ are independent of $d$.
\end{rem}

Preceding to the next argument, we introduce an expression which is used in the rest of this section. For any index $\mathfrak{s}=(s_1, s_2, \ldots, s_r)\in\mathbb{N}^r$, we can write $\mathfrak{s}=(s_1, \mathfrak{s}')$ where $\mathfrak{s}'=(s_2, \ldots, s_r)\in\mathbb{N}^r$ (resp. $\boldsymbol{\epsilon}=(\epsilon_1, \ldots, \epsilon_r)\in(\mathbb{F}_q^{\times})^{r}$) and when $r=1$, we set $\mathfrak{s}'=\phi$ (resp. $\boldsymbol{\epsilon}'=\phi$) and further $S_d(\mathfrak{s}';\boldsymbol{\epsilon}'):=1$.

Next we prepare the following lemma to show sum-shuffle relations for alternating power sums in general depth.  
\begin{lem}\label{lempowless}
For ${\mathfrak a}:=(a_1, a_2, \ldots, a_r)\in\mathbb{N}^{r}$, ${\mathfrak b}:=(b_1, b_2, \ldots, b_s)\in\mathbb{N}^{s}$, ${\boldsymbol \epsilon}:=(\epsilon_1, \epsilon_2, \ldots, \epsilon_r)\in(\mathbb{F}_q^{\times})^{r}$ and ${\boldsymbol \lambda}:=(\lambda_1, \lambda_2, \ldots, \lambda_s)\in(\mathbb{F}_q^{\times})^{s}$, we may express the product $S_{<d}({\mathfrak a};{\boldsymbol \epsilon})S_{<d}({\mathfrak b};{\boldsymbol \lambda})$ as follows:
\begin{align}\label{lempow}
	S_{<d}({\mathfrak a};{\boldsymbol \epsilon})S_{<d}({\mathfrak b}&;{\boldsymbol \lambda})=\sum_{i} {f_i}S_{<d}(c_{i1}, \ldots, c_{i{l_{i}}}; \mu_{i1}, \ldots, \mu_{i{l_{i}}})
\end{align}
for some $c_{ij}\in\mathbb{N}$, $\mu_{ij}\in\mathbb{F}_q^{\times}$ so that $\sum_{m=1}^{r}a_m+\sum_{n=1}^{s}b_n=\sum_{h=1}^{l_i}c_h$, $\prod_{m=1}^{r}\epsilon_m\prod_{n=1}^{s}\lambda_n=\prod_{h=1}^{l_i}\mu_h$, $l_i\leq r+s$ and $f_i\in\mathbb{F}_p$ which are independent of $d$ for each $i$. 
\end{lem}
\begin{proof}
We proceed to prove the result by induction on ${\rm dep}(\mathfrak{a})+{\rm dep}(\mathfrak{b})(=r+s)>2$:
\begin{align*}
	S_{<d}&({\mathfrak a};{\boldsymbol \epsilon})S_{<d}({\mathfrak b};{\boldsymbol \lambda})\\
	&=\biggl(\sum_{0\leq m_1<d}S_{m_1}({\mathfrak a};{\boldsymbol \epsilon})\biggr)\biggl(\sum_{0\leq n_1<d}S_{n_1}({\mathfrak b};{\boldsymbol \lambda})\biggr)=\sum_{0\leq m_1, n_1<d}S_{m_1}({\mathfrak a};{\boldsymbol \epsilon})S_{n_1}({\mathfrak b};{\boldsymbol \lambda})\\
 &=\sum_{0\leq n_1<m_1<d }S_{m_1}({\mathfrak a};{\boldsymbol \epsilon})S_{n_1}({\mathfrak b};{\boldsymbol \lambda})+\sum_{0\leq m_1<n_1<d}S_{m_1}({\mathfrak a};{\boldsymbol \epsilon})S_{n_1}({\mathfrak b};{\boldsymbol \lambda})\\
    &\quad+\sum_{0\leq n_1=m_1<d}S_{m_1}({\mathfrak a};{\boldsymbol \epsilon})S_{n_1}({\mathfrak b};{\boldsymbol \lambda})\\
	&=\sum_{0\leq m_1<d}\biggl(S_{m_1}(a_1;\epsilon_1)S_{<m_1}({\mathfrak a}';{\boldsymbol \epsilon}')\sum_{0\leq n_1<m_1}S_{n_1}({\mathfrak b};{\boldsymbol \lambda})\biggr)\\
	&\quad+\sum_{0\leq n_1<d}\biggl(S_{n_1}(b_1;\lambda_1)S_{<n_1}({\mathfrak b}';{\boldsymbol \lambda}')\sum_{0\leq m_1<n_1}S_{m_1}({\mathfrak a};{\boldsymbol \epsilon})\biggr)\\
	&\quad+\sum_{0\leq m_1=n_1<d}S_{m_1}(a_1;\epsilon_1)S_{<m_1}({\mathfrak a}';{\boldsymbol \epsilon}')S_{n_1}(b_1;\lambda_1)S_{<n_1}({\mathfrak b}';{\boldsymbol \lambda}')\\
    &=\sum_{0\leq m_1<d}S_{m_1}(a_1;\epsilon_1)S_{<m_1}({\mathfrak a}';{\boldsymbol \epsilon}')S_{<m_1}({\mathfrak b};{\boldsymbol \lambda})\\
	&\quad+\sum_{0\leq n_1<d}S_{n_1}(b_1;\lambda_1)S_{<n_1}({\mathfrak b}';{\boldsymbol \lambda}')S_{<n_1}({\mathfrak a};{\boldsymbol \epsilon})\\
	&\quad+\sum_{0\leq m_1<d}S_{m_1}(a_1;\epsilon_1)S_{<m_1}({\mathfrak a}';{\boldsymbol \epsilon}')S_{m_1}(b_1;\lambda_1)S_{<m_1}({\mathfrak b}';{\boldsymbol \lambda}')\\
\intertext{by applying induction hypothesis to the first, second series of power sums and applying Lemma \ref{hl1} to the third series in the right hand side of last equality, we have}
  \end{align*}
	\begin{align*} 
    &=\sum_{0\leq m_1<d}S_{m_1}(a_1;\epsilon_1)\sum_{i} {f_i}S_{<m_1}(c_{i1}, \ldots, c_{il_{i}}; \mu_{i1}, \ldots, \mu_{il_{i}})\\
	&\quad+\sum_{0\leq n_1<d}S_{n_1}(b_1;\lambda_1)\sum_{i} {f'_i}S_{<n_1}(c'_{i1}, \ldots, c'_{il_{i}}; \mu'_{i1}, \ldots, \mu'_{il_{i}})\\
	&\quad+\sum_{0\leq m_1<d}\Biggl(\sum_{\substack{0<j<a_1+b_1\\q-1|j}}\Delta^{j}_{a_1, b_1}S_{m_1}(a_1+b_1-j, j; \epsilon_1\lambda_1, 1)+S_{m_1}(a_1+b_1; \epsilon_1\lambda_1)\Biggr)\\
	&\qquad \qquad \qquad\cdot S_{<m_1}({\mathfrak a}';{\boldsymbol \epsilon}')S_{<m_1}({\mathfrak b}';{\boldsymbol \lambda}')\\
    &=\sum_{0\leq m_1<d}S_{m_1}(a_1;\epsilon_1)\sum_{i} {f_i}S_{<m_1}(c_{i1}, \ldots, c_{il_{i}}; \mu_{i1}, \ldots, \mu_{il_{i}})\\
	&\quad+\sum_{0\leq n_1<d}S_{n_1}(b_1;\lambda_1)\sum_{i} {f'_i}S_{<n_1}(c'_{i1}, \ldots, c'_{il_{i}}; \mu'_{i1}, \ldots, \mu'_{il_{i}})\\
	&\quad+\sum_{0\leq m_1<d}\Biggl(\sum_{\substack{0<j<a_1+b_1\\q-1|j}}\Delta^{j}_{a_1, b_1}S_{m_1}(a_1+b_1-j; \epsilon_1\lambda_1)S_{<m_1}(j;1)+S_{m_1}(a_1+b_1; \epsilon_1\lambda_1)\Biggr)\\
	&\qquad \qquad \qquad\cdot S_{<m_1}({\mathfrak a}';{\boldsymbol \epsilon}')S_{<m_1}({\mathfrak b}';{\boldsymbol \lambda}').	
	\intertext{Again by applying induction hypothesis for $ S_{<m_1}({\mathfrak a}';{\boldsymbol \epsilon}')S_{<m_1}({\mathfrak b}';{\boldsymbol \lambda}')$, the above quantity equals }
	&=\sum_{0\leq m_1<d}\sum_{i}{f_i}S_{m_1}(a_1, c_{i1}, \ldots, c_{il_{i}};\epsilon_1, \mu_{i1}, \ldots, \mu_{il_{i}})\\
	&\quad+\sum_{0\leq n_1<d}\sum_{i}{f'_i}S_{n_1}(b_1, c'_{i1}, \ldots, c'_{il_{i}};\lambda_1, \mu'_{i1}, \ldots, \mu'_{il_{i}})\\
	&\quad+\sum_{0\leq m_1<d}\sum_{\substack{0<j<a_1+b_1\\q-1|j}}\Delta^{j}_{a_1, b_1}\sum_{i}g_iS_{m_1}(a_1+b_1-j, e_{i1}, \ldots, e_{il_i}; \epsilon_1\lambda_1, \eta_{i1}, \ldots, \eta_{il_i})\\
	&+\sum_{0 \leq m_1<d}\sum_{i}g'_iS_{m_1}(a_1+b_1, e_{i1}', \ldots, e_{il_i}'; \epsilon_1\lambda_1,  \eta_{i1}', \ldots, \eta_{il_i}')
\end{align*}
for some $g_i$, $g_i'\in\mathbb{F}_p$. 
Therefore the product $S_{<d}({\mathfrak a};{\boldsymbol \epsilon})S_{<d}({\mathfrak b};{\boldsymbol \lambda})$ is expressed by an $\mathbb{F}_p$-linear combination of $S_{<d}(- ;-)$ with the desired conditions. 
\end{proof}
To prove the sum-shuffle result in Theorem \ref{shamz} the following is the key ingredient.
\begin{thm}\label{powsha}
For ${\mathfrak a } :=(a_1, a_2, \ldots, a_r)\in\mathbb{N}^{r}$, ${\mathfrak b}:=(b_1, b_2, \ldots, b_s)\in\mathbb{N}^{s}$, ${\boldsymbol \epsilon}:=(\epsilon_1, \epsilon_2, \ldots, \epsilon_r)\in(\mathbb{F}_q^{\times})^{r}$ and ${\boldsymbol \lambda}:=(\lambda_1, \lambda_2, \ldots, \lambda_s)\in(\mathbb{F}_q^{\times})^{s}$, we may express the product $S_d({\mathfrak a }; {\boldsymbol \epsilon})S_d({\mathfrak b} ;{\boldsymbol \lambda})$ as follows:
\begin{align}\label{shamzv}
	S_{d}({\mathfrak a};{\boldsymbol \epsilon})S_{d}({\mathfrak b};{\boldsymbol \lambda})=\sum_{i} f'_iS_d(c_{i1}, \ldots, c_{il_{i}}; \mu_{i1}, \ldots, \mu_{il_{i}})
\end{align}
for some $c_{ij}\in\mathbb{N}$, $\mu_{ij}\in\mathbb{F}_q^{\times}$ so that $\sum_{m=1}^{r}a_m+\sum_{n=1}^{s}b_n=\sum_{h=1}^{l_i}c_{ih}$, $\prod_{m=1}^{r}\epsilon_m\prod_{n=1}^{s}\lambda_n=\prod_{h=1}^{l_i}\mu_{ih}$, $l_i\leq r+s$ and $f'_i\in\mathbb{F}_p$ for each $i$. 
\end{thm}
\begin{proof}
From the decomposition which is described in \eqref{extpowsum}, we have
\[
S_d({\mathfrak a};{\boldsymbol \epsilon})S_d({\mathfrak b};{\boldsymbol\lambda})=S_d(a_1;\epsilon_1)S_{<d}({\mathfrak a}';{\boldsymbol \epsilon}')S_d(b_1;\lambda_1)S_{<d}({\mathfrak b}';{\boldsymbol \lambda}').
\]
\begin{align*}
\intertext{By using the equation \eqref{lempow} to the product $S_{<d}({\mathfrak a}';{\boldsymbol \epsilon}')S_{<d}({\mathfrak b}';{\boldsymbol \lambda}')$, we have}
S_d({\mathfrak a};{\boldsymbol \epsilon})S_{d}({\mathfrak b};{\boldsymbol \lambda})&=S_d(a_1;\epsilon_1)S_d(b_1;\lambda_1)\Biggl(\sum_{i} {f_i}S_{<d}(c_{i1}, \ldots, c_{i{l_{i}}}; \mu_{i1}, \ldots, \mu_{i{l_{i}}})\Biggr)
\intertext{Using the equation \eqref{lempowsh}, the quantity above is equal to}
&=\sum_{\substack{0<j<a_1+b_1\\q-1|j}}\Biggl(\Delta^{j}_{a_1, b_1}S_{d}(a_1+b_1-j, j; \epsilon_1\lambda_1, 1)+S_{d}(a_1+b_1; \epsilon_1\lambda_1)\Biggr)\\
&\qquad \cdot\Biggl(\sum_{i} {f_i}S_{<d}(c_{i1}, \ldots, c_{i{l_{i}}}; \mu_{i1}, \ldots, \mu_{i{l_{i}}})\Biggr)\\
&=\sum_{\substack{0<j<a_1+b_1\\q-1|j}}\Biggl(\Delta^{j}_{a_1, b_1}S_{d}(a_1+b_1-j;\epsilon_1\lambda_1)S_{<d}(j;1)+S_{d}(a_1+b_1; \epsilon_1\lambda_1)\Biggr)\\
&\qquad \cdot\Biggl(\sum_{i} {f_i}S_{<d}(c_{i1}, \ldots, c_{i{l_{i}}}; \mu_{i1}, \ldots, \mu_{i{l_{i}}})\Biggr)\\
&=\sum_{\substack{0<j<a_1+b_1\\q-1|j}}\sum_{i} {f_i}\Delta^{j}_{a_1, b_1}S_{d}(a_1+b_1-j;\epsilon_1\lambda_1)\\
&\qquad\cdot S_{<d}(j;1)S_{<d}(c_{i1}, \ldots, c_{i{l_{i}}}; \mu_{i1}, \ldots, \mu_{i{l_{i}}})\\
&\quad +\sum_{\substack{0<j<a_1+b_1\\q-1|j}}\sum_{i} {f_i}S_{d}(a_1+b_1; \epsilon_1\lambda_1)S_{<d}(c_{i1}, \ldots, c_{i{l_{i}}}; \mu_{i1}, \ldots, \mu_{i{l_{i}}}).\\
\intertext{Applying \eqref{lempow} to the product $S_{<d}(j;1)S_{<d}(c_{i1}, \ldots, c_{i{l_{i}}}; \mu_{i1}, \ldots, \mu_{i{l_{i}}})$, we have}
S_d({\mathfrak a};{\boldsymbol \epsilon})S_{d}({\mathfrak b};{\boldsymbol \lambda})&=\sum_{\substack{0<j<a_1+b_1\\q-1|j}}\sum_{i} {f_i}\Delta^{j}_{a_1, b_1}S_{d}(a_1+b_1-j;\epsilon_1\lambda_1)\\
&\qquad\cdot\Biggl(\sum_{m}g_mS_{<d}(e_{m1}, \ldots, e_{m{n_{m}}}; \nu_{m1}, \ldots, \nu_{m{n_{m}}})\Biggr)\\
&\quad +\sum_{\substack{0<j<a_1+b_1\\q-1|j}}\sum_{i} {f_i}S_{d}(a_1+b_1; \epsilon_1\lambda_1)S_{<d}(c_{i1}, \ldots, c_{i{l_{i}}}; \mu_{i1}, \ldots, \mu_{i{l_{i}}})\\
&=\sum_{\substack{0<j<a_1+b_1\\q-1|j}}\sum_{i, m} f_ig_m\Delta^{j}_{a_1, b_1}S_{d}(a_1+b_1-j, e_{m1}, \ldots, e_{m{n_{m}}};\epsilon_1\lambda_1, \nu_{m1}, \ldots, \nu_{m{n_{m}}})\\
&\quad +\sum_{\substack{0<j<a_1+b_1\\q-1|j}}\sum_{i}{f_i}S_{d}(a_1+b_1, c_{i1}, \ldots, c_{i{l_{i}}};\epsilon_1\lambda_1, \mu_{i1}, \ldots, \mu_{i{l_{i}}}).
\end{align*}
Thus we show that $S_{d}({\mathfrak a};{\boldsymbol \epsilon})S_{d}({\mathfrak b};{\boldsymbol \lambda})$ is expressed by  with the desired formulation.
\end{proof}

\begin{rem}
In Theorem \ref{powsha}, coefficients $f_i\in\mathbb{F}_p$ are independent of $d$ by Lamma \ref{lempowless}.
\end{rem}

The main result on sum-shuffle relations for AMZVs are the following.
\begin{thm}\label{shamz}
For ${\mathfrak a } :=(a_1, a_2, \ldots, a_r)\in\mathbb{N}^{r}$, ${\mathfrak b}:=(b_1, b_2, \ldots, b_s)\in\mathbb{N}^{s}$, ${\boldsymbol \epsilon}:=(\epsilon_1, \epsilon_2, \ldots, \epsilon_r)\in(\mathbb{F}_q^{\times})^{r}$ and ${\boldsymbol \lambda}:=(\lambda_1, \lambda_2, \ldots, \lambda_s)\in(\mathbb{F}_q^{\times})^{s}$, we may express the product $\zeta_A({\mathfrak a};{\boldsymbol \epsilon})\zeta_A({\mathfrak b};{\boldsymbol \lambda})$ as follows:
\begin{align}\label{shaamzv}
	\zeta_A({\mathfrak a};{\boldsymbol \epsilon})\zeta_A({\mathfrak b};{\boldsymbol \lambda})=\sum_{i} f''_i\zeta_A(c_{i1}, \ldots, c_{il_{i}}; \mu_{i1}, \ldots, \mu_{il_{i}})
\end{align}
for some $c_{ij}\in\mathbb{N}$ and $\mu_{ij}\in\mathbb{F}_q^{\times}$ so that $\sum_{m=1}^{r}a_m+\sum_{n=1}^{s}b_n=\sum_{h=1}^{l_i}c_{ih}$, $\prod_{m=1}^{r}\epsilon_m\prod_{n=1}^{s}\lambda_n=\prod_{h=1}^{l_i}\mu_{ih}$, $l_i\leq r+s$ and $f''_i\in\mathbb{F}_p$ for each $i$. 
\end{thm}
\begin{proof}
By the set theoretical inclusion-exclusion principle, we can express $\sum_{d\geq0}S_{d}({\mathfrak a};{\boldsymbol \epsilon})S_{d}({\mathfrak b};{\boldsymbol \lambda})$ as follows:
\begin{align*}
	&\sum_{d\geq0}S_{d}({\mathfrak a};{\boldsymbol \epsilon})S_{d}({\mathfrak b};{\boldsymbol \lambda})\\
	&=\Biggl(\sum_{d_1\geq 0}S_{d_1}({\mathfrak a};{\boldsymbol \epsilon})\Biggr)\Biggl(\sum_{e_1\geq0}S_{e_1}({\mathfrak b};{\boldsymbol \lambda})\Biggr)-\sum_{d_1>e_1\geq0}S_{d_1}({\mathfrak a};{\boldsymbol \epsilon})S_{e_1}({\mathfrak b};{\boldsymbol \lambda})-\sum_{e_1>d_1\geq0}S_{e_1}({\mathfrak b};{\boldsymbol \lambda})S_{d_1}({\mathfrak a};{\boldsymbol \epsilon})\\
	&=\Biggl(\sum_{d_1\geq 0}S_{d_1}({\mathfrak a};{\boldsymbol \epsilon})\Biggr)\Biggl(\sum_{e_1\geq0}S_{e_1}({\mathfrak b};{\boldsymbol \lambda})\Biggr)-\sum_{d_1>0}S_{d_1}({\mathfrak a};{\boldsymbol \epsilon})S_{<d_1}({\mathfrak b};{\boldsymbol \lambda})-\sum_{e_1>0}S_{e_1}({\mathfrak b};{\boldsymbol \lambda})S_{<e_1}({\mathfrak a};{\boldsymbol \epsilon})\\
    &=\Biggl(\sum_{d_1\geq 0}S_{d_1}({\mathfrak a};{\boldsymbol \epsilon})\Biggr)\Biggl(\sum_{e_1\geq0}S_{e_1}({\mathfrak b};{\boldsymbol \lambda})\Biggr)\\
    &\quad-\sum_{d_1>0}S_{d_1}(a_1;\epsilon_1)S_{<d_1}({\mathfrak a}';{\boldsymbol \epsilon}')S_{<d_1}({\mathfrak b};{\boldsymbol \lambda})-\sum_{e_1>0}S_{e_1}(b_1;\lambda_1)S_{<e_1}({\mathfrak b}';{\boldsymbol \lambda}')S_{<e_1}({\mathfrak a};{\boldsymbol \epsilon}).
    \intertext{By \eqref{lempow}, the above is equal to}
    &=\Biggl(\sum_{d_1\geq 0}S_{d_1}({\mathfrak a};{\boldsymbol \epsilon})\Biggr)\Biggl(\sum_{e_1\geq0}S_{e_1}({\mathfrak b};{\boldsymbol \lambda})\Biggr)-\sum_{d_1>0}\sum_{i} {g_i}S_{d_1}(a_1, m_{i1}, \ldots, m_{i{l_{i}}}; \epsilon_1, \mu_{i1}, \ldots, \mu_{i{l_{i}}})\\
    &\hspace{5.13cm}  -\sum_{e_1>0}\sum_{i} {h_i}S_{e_1}(b_1, n_{i1}, \ldots, n_{i{l_{i}}};\lambda_1, \eta_{i1}, \ldots, \eta_{i{l_{i}}}).
\end{align*}
for some $g_i$, $h_i\in\mathbb{F}_p$. Combining the expression of $S_d({\mathfrak a};{\boldsymbol \epsilon})S_d({\mathfrak b};{\boldsymbol \lambda} )$ in \eqref{shamzv} with the above, we have the following identity
\begin{align*}
\sum_{d\geq0}\sum_{i} f'_iS_d(c_{i1}, \ldots, c_{il_{i}}; \mu_{i1}, \ldots, \mu_{il_{i}})&=\sum_{d_1\geq 0}S_{d_1}({\mathfrak a};{\boldsymbol \epsilon})\sum_{e_1\geq0}S_{e_1}({\mathfrak b};{\boldsymbol \lambda})\\
    &\ -\sum_{d_1>0}\sum_{i} {g_i}S_{d_1}(a_1, m_{i1}, \ldots, m_{i{l_{i}}}; \epsilon_1, \mu_{i1}, \ldots, \mu_{i{l_{i}}})\\
    &\ -\sum_{e_1>0}\sum_{i} {h_i}S_{e_1}(b_1, n_{i1}, \ldots, n_{i{l_{i}}};\lambda_1, \eta_{i1}, \ldots, \eta_{i{l_{i}}}).
\end{align*}
The coefficients $f'_i$, $g_i$, $h_i$ are independent of $d$ by Theorem \ref{powsha}. Thus by \eqref{chpes}, we obtain 
\begin{align*}
\sum_{i} f'_i\zeta_A(c_{i1}, \ldots, c_{il_{i}}; \mu_{i1}, \ldots, \mu_{il_{i}})&=\zeta_A({\mathfrak a};{\boldsymbol \epsilon})\zeta_A({\mathfrak b};{\boldsymbol \lambda})\\
    &\ -\sum_{i} {g_i}\zeta_A(a_1, m_{i1}, \ldots, m_{i{l_{i}}}; \epsilon_1, \mu_{i1}, \ldots, \mu_{i{l_{i}}})\\
    &\ -\sum_{i} {h_i}\zeta_A(b_1, n_{i1}, \ldots, n_{i{l_{i}}};\lambda_1, \eta_{i1}, \ldots, \eta_{i{l_{i}}}).
\end{align*}
Thus we have 
\begin{align*}
\zeta_A({\mathfrak a};{\boldsymbol \epsilon})\zeta_A({\mathfrak b};{\boldsymbol \lambda})&=\sum_{i} f'_i\zeta_A(c_{i1}, \ldots, c_{il_{i}}; \mu_{i1}, \ldots, \mu_{il_{i}})\\
    &\quad +\sum_{i} {g_i}\zeta_A(a_1, m_{i1}, \ldots, m_{i{l_{i}}}; \epsilon_1, \mu_{i1}, \ldots, \mu_{i{l_{i}}})\\
    &\quad +\sum_{i} {h_i}\zeta_A(b_1, n_{i1}, \ldots, n_{i{l_{i}}};\lambda_1, \eta_{i1}, \ldots, \eta_{i{l_{i}}}).
\end{align*}
By Lemma \ref{lempowless} and Theorem \ref{powsha}, the indices and coefficients in the above equation satisfy desired conditions. 

Therefore we obtain the sum-shuffle relation for $\zeta_A({\mathfrak s}; {\boldsymbol \epsilon})$.
\end{proof}
By Theorem \ref{shamz}, the $\mathbb{F}_p$-linear span of our AMZVs form an $\mathbb{F}_p$-algebra. 
\begin{eg}
For some $q$, $({\mathfrak a};{\boldsymbol \epsilon})$ and $({\mathfrak b};{\boldsymbol \lambda})$, the product $\zeta_A({\mathfrak a};{\boldsymbol \epsilon})\zeta_A({\mathfrak b};{\boldsymbol \lambda})$ is explicitly computed as follows.  

When $q=3$, $({\mathfrak a};{\boldsymbol \epsilon})=(2;1)$ and $({\mathfrak b};{\boldsymbol \lambda})=(1,2;2,2)$,
\begin{align*}
\zeta_A(2;1)\zeta_A(1,2;2,2)&=\zeta_A(3,2;2,2)-\zeta_A(1,2,2;2,2,1)+\zeta_A(1,2,2;2,1,2)\\
&\quad +\zeta_A(1,4;2,2)+\zeta_A(2,1,2;1,2,2)
\end{align*}
When $q=5$, $({\mathfrak a};{\boldsymbol \epsilon})=(2;3)$ and $({\mathfrak b};{\boldsymbol \lambda})=(3;1)$,
\begin{align*}
\zeta_A(2;3)\zeta_A(3;1)=\zeta_A(5;3)+\zeta_A(2,3;3,1)+\zeta_A(3,2;1,3).
\end{align*}
\end{eg}

Later we give an explicit sum-shuffle relation for a product of any depth 2 AMZV and depth 1 AMZV in Appendix.

\section{Period interpretation of AMZVs}\label{secperint}
In this section, we show that each $\zeta_A({\mathfrak s};{\boldsymbol \epsilon})$ appears as a period of certain pre-$t$-motive basing on the idea in \cite{AT1}. 

We denote the ring $\overline{k}(t)[\sigma, \sigma^{-1}]$ the non-commutative Laurent polynomial ring in $\sigma$ with coefficients in $\overline{k}(t)$ subject to the following relations,
\[
 \sigma f=f^{(-1)}\sigma, \quad f\in\overline{k}(t).
\]

We denote $\mathbb{E}$ to be the ring consisting of formal power series 
\[
    \sum_{n=0}^{\infty}a_nt^n\in\overline{k}[[t]]
\]
such that 
\[
    \lim_{n\rightarrow \infty}\sqrt[n]{|a_n|_{\infty}}=0,\quad [k_{\infty}(a_0, a_1, \ldots):k_{\infty}]<\infty.  
\]
We note that the former condition guarantees that such a series is an entire function, that is, it has an infinite radius of convergence with respect to the absolute value $|\cdot|_{\infty}$ thus $\mathbb{E}\subset\mathbb{T}$. The latter condition guarantees that for any $t_0\in\overline{k_{\infty}}$ the value of such a series at $t=t_0$ belongs again to $\overline{k_{\infty}}$.
We note that
\begin{align}\label{omegaentire}
	\Omega\in\mathbb{E}.
\end{align}
Therefore $\Omega\in\mathbb{T}$.

\subsection{Review of pre-t-motive}
First we recall the notions of pre-$t$-motives.   
For more detail, see \cite{Pa}.
\begin{defn}[{\cite[\S 3.2.1]{Pa}}]
A {\it pre-$t$-motive} $M$ is a left $\overline{k}(t)[\sigma, \sigma^{-1}]$-module that is finite dimensional over $\overline{k}(t)$.
\end{defn}

Let $M$ be a pre-$t$-motive dimension $r$ over $\overline{k}(t)$ and $\Phi$ be representing matrix of multiplication by $\sigma$ on $M$ with respect to a given basis of {\bf m} of $M$, then $M$ is rigid analytically trivial if and only if there is a matrix $\Psi\in {\rm GL}_r(\mathbb{L})$ ($\mathbb{L}$ is a quotient field of $\mathbb{T}$) satisfying
\[
	\Psi^{(-1)}=\Phi\Psi.
\]
Here we define $\Psi^{(-1)}$ by $(\Psi^{(-1)})_{ij}:=(\Psi_{ij})^{(-1)}$.
Such matrix $\Psi$ is called {\it rigid analytic trivialization} of $\Phi$ and if all the entries of $\Psi$ are convergent at $t=\theta$, and entries of $\Psi|_{t=\theta}$ are called {\it periods} of $M$.

Anderson and Thakur obtained period interpretation of MZVs in \cite{AT1}, that is, they showed that MZVs appear as entries of $\Psi|_{t=\theta}$ for $\Psi$ which is a rigid analytic trivialization of the rigid analytically trivial pre-$t$-motive with dimension $r+1$ over $\overline{k}(t)$ multiplication by $\sigma$ is represented by the following matrix $\Phi$:
\begin{align*}
\Phi:=
\begin{array}{rccccccll}
\ldelim({5.5}{4pt}[] & (t-\theta)^{s_1+\cdots+s_r} & 0 & 0 & \cdots & 0 & \rdelim){5.5}{4pt}[] &\\
&(t-\theta)^{s_1+\cdots+s_r}H_{s_1-1}^{(-1)} & (t-\theta)^{s_2+\cdots+s_r} & 0 & \cdots & 0 & &\\
& 0 &(t-\theta)^{s_2+\cdots+s_r}H_{s_2-1}^{(-1)} & \ddots & & \vdots & &\\
&\vdots &  & \ddots & (t-\theta)^{s_r} & 0 & &\\
& 0 & \cdots & 0 &(t-\theta)^{s_r}H_{s_r-1}^{(-1)} & 1 & &.\\
\end{array}
\end{align*}
They obtained the rigid analytic trivialization of $\Phi$ as the following matrix:
\[
	\Psi:=
\begin{array}{rccccccll}
\ldelim({6}{4pt}[] &\Omega^{s_1+\cdots+s_r} &  &  &  &  & & \rdelim){6}{4pt}[] & \\
&L(s_1)\Omega^{s_2+\cdots+s_r} & \Omega^{s_2+\cdots+s_r} &  &  &  &  & &\\
&\vdots &L(s_2)\Omega^{s_3+\cdots+s_r} & \ddots &  &  & & &\\
&\vdots & \vdots & \ddots & \ddots & & & &\\
&L(s_1, \ldots, s_{r-1})\Omega^{s_r} & L(s_2, \ldots. s_{r-1})\Omega^{s_r} &  & \ddots & \Omega^{s_r} & & &\\
&L(s_1, \ldots, s_r)& L(s_2, \ldots, s_r) & \cdots & \cdots & L(s_r) & 1 & &.
\end{array}
\]

Here we define
\begin{align*}
	L(s_1, \ldots, s_r)&:=\sum_{d_1>\cdots>d_r\geq 0}(H_{s_1-1}\Omega^{s_1})^{(d_1)}\cdots(H_{s_r-1}\Omega^{s_r})^{(d_r)}.
\end{align*}

By using $\Omega|_{t=\theta}=\tilde{\pi}^{-1}$ and \eqref{atint}, we obtain the following for $1\leq i\leq j\leq r$:
\begin{align*}
	L(s_i, \ldots, s_j)|_{t=\theta}&=\frac{\Gamma_{s_i}\cdots\Gamma_{s_j}}{\tilde{\pi}^{s_i+\cdots+s_j}}\sum_{d_i>\cdots >d_j\geq0}S_{d_i}(s_i)\cdots S_{d_j}(s_j)=\frac{\Gamma_{s_i}\cdots\Gamma_{s_j}}{\tilde{\pi}^{s_i+\cdots+s_j}}\zeta_A(s_i, \ldots, s_j).\nonumber
\end{align*}
Therefore it is known that MZVs are periods of the rigid analytically trivial pre-$t$-motive $M$ defined by $\Phi$.   
\subsection{AMZVs as periods}
Next we give a period interpretation of AMZVs . 
\begin{defn}\label{M}
Given $\gamma_i\in\overline{\mathbb{F}_q}^{\times}$ $(i=1, \ldots, r)$, a fixed $(q-1)$st root of $\epsilon_i\in\mathbb{F}_q^{\times}$, we let $M$ be the pre-$t$-motive such that $\dim_{\overline{k}(t)}M=r+1$ and for the fixed $\overline{k}(t)$-basis ${\bf m}$ of $M$ which satisfy
\begin{align*}
	\sigma{\bf m}=\Phi{\bf m}
\end{align*} 
where $\Phi$ is the following matirix:
\begin{align*}
\begin{array}{rcccccll}
\ldelim({5.5}{4pt}[] & (t-\theta)^{s_1+\cdots+s_r} & 0 & 0 & \cdots & 0 &\rdelim){5.5}{4pt}[] &\\
&\gamma_1^{(-1)}(t-\theta)^{s_1+\cdots+s_r}H_{s_1-1}^{(-1)} & (t-\theta)^{s_2+\cdots+s_r} & 0 & \cdots & 0 & &\\
& 0 & \gamma_2^{(-1)}(t-\theta)^{s_2+\cdots+s_r}H_{s_2-1}^{(-1)} & \ddots & & \vdots & & \\
&\vdots &  & \ddots & (t-\theta)^{s_r} & 0 & &\\
&0 & \cdots & 0 & \gamma_r^{(-1)}(t-\theta)^{s_r}H_{s_r-1}^{(-1)} & 1 & &.\\
\end{array}
\end{align*}
Here $H_{s_i-1}\in A[t]$ is the Anderson-Thakur polynomial.
\end{defn}
For the above pre-$t$-motive, we can give a matrix $\Psi$ which has a relation with $\Phi$ as the following proposition.
\begin{prop}\label{rat}
For the pre-$t$-motive $M$ defined in Definition \ref{M}, there exists $\Psi\in{\rm GL}_{r+1}(\overline{k}[t])$ so that it satisfies $\Psi^{(-1)}=\Phi\Psi$. 
\end{prop}

\begin{proof}
For the matrix $\Phi$ in Definition \ref{M}, we show that the matrix $\Psi\in{\rm GL}_{r+1}(\overline{k}[t])$ satisfies $\Psi^{(-1)}=\Phi\Psi$ and
which is given as follows:
\begin{align}\label{Psi}
\Psi:=
&\begin{array}{rccccccll}
\ldelim({6}{1pt}[] &\Omega^{s_1+\cdots+s_r} &  &  &  &  & &
\rdelim){6}{1pt}[]  &\\
&\gamma_1L(s_1)\Omega^{s_2+\cdots+s_r} & \Omega^{s_2+\cdots+s_r} &  &
&  & & & \\
&\vdots & \gamma_2L(s_2)\Omega^{s_3+\cdots+s_r} & \ddots &  &  & & & \\
&\vdots & \vdots & \ddots & \ddots & & & &\\
&\gamma_1\cdots\gamma_{r-1}L(s_1, \ldots, s_{r-1})\Omega^{s_r} &
\gamma_2\cdots\gamma_{r-1}L(s_2, \ldots. s_{r-1})\Omega^{s_r} &  &
\ddots & \Omega^{s_r} & & &\\
&\gamma_1\cdots\gamma_{r}L(s_1, \ldots, s_r)&
\gamma_2\cdots\gamma_{r}L(s_2, \ldots, s_r) & \cdots & \cdots &
\gamma_rL(s_r) &1 & &.
\end{array}
\end{align}
Here we define the following:
\begin{align*}
	L(s_1, \ldots, s_r)&:=L(s_1, \ldots, s_r; \epsilon_1, \ldots, \epsilon_r)\\
				   &:=\sum_{d_1>\cdots>d_r\geq 0}\epsilon_1^{d_1}(H_{s_1-1}\Omega^{s_1})^{(d_1)}\cdots\epsilon_r^{d_r}(H_{s_r-1}\Omega^{s_r})^{(d_r)}.
\end{align*}
We also note that the following equations hold by their definitions.
\begin{align}
\gamma_i^{(-1)}=&\epsilon_i^{-1}\gamma_i, \label{deqg}\\
\Omega^{(-1)}=&(t-\theta)\Omega, \label{deqo}\\
L(s_1, \ldots, s_r)^{(-1)}=&\biggl(\prod_{n=0}^{r}\epsilon_n\biggr)L(s_1, \ldots, s_r)\label{deql}\\
&+\biggl(\prod_{n=0}^{r-1}\epsilon_n\biggr)(\Omega^{s_r}H_{s_r-1})^{(-1)}L(s_1, \ldots, s_{r-1}). \nonumber
\end{align}

Now we show the validity of the equation 
\[
\Psi^{(-1)}=\Phi\Psi
\]
by comparing their entries.

For $1\leq i\leq r+1$, and $i\leq j\leq r+1$, it is obvious that $(i, j)$-th entry of the left hand side $\Psi^{(-1)}_{ij}$ is equal to the $(i, j)$-th entry of the right hand side $(\Phi\Psi)_{ij}$ by \eqref{deqo}.

In the rest of this proof, we set $\prod_{n=i}^j\gamma_{n}=1$ for $0\leq j<i$ since it is empty product. 
For $2\leq i\leq r$ and $1\leq j < i$, $\Psi^{(-1)}_{ij}$ is transformed as follows.
\begin{align*}
\Psi^{(-1)}_{ij}=&\biggl(\prod_{n=j}^{i-1}(\gamma_n)^{(-1)}\biggr)(t-\theta)^{s_i+\cdots+s_r}\Omega^{s_i+\cdots+s_r}L(s_j, \ldots, s_{i-1})^{(-1)}\nonumber\\
                  \intertext{by using the equation \eqref{deql},}
                  =&\biggl(\prod_{n=j}^{i-1}\gamma_n^{(-1)}\biggr)\bigl\{(t-\theta)\Omega\bigr\}^{s_i+\cdots+s_r}\biggl(\prod_{n=j}^{i-1}\epsilon_n\biggr) L(s_j, \ldots, s_{i-1})\nonumber\\
                  &+\biggl(\prod_{n=j}^{i-1}\gamma_n^{(-1)}\biggr)\bigl\{(t-\theta)\Omega\bigr\}^{s_i+\cdots+s_r}\biggl(\prod_{n=j}^{i-2}\epsilon_n\biggr)(\Omega^{s_{i-1}}H_{s_{i-1}-1})^{(-1)}L(s_j, \ldots, s_{i-2})\nonumber
                  \intertext{by using the equation \eqref{deqg},}
                  =&\biggl(\prod_{n=j}^{i-1}\gamma_n\biggr)\bigl\{(t-\theta)\Omega\bigr\}^{s_i+\cdots+s_r}L(s_j, \ldots, s_{i-1})\nonumber\\
                  &+\gamma_{i-1}^{(-1)}\biggl(\prod_{n=j}^{i-2}\gamma_{n}\biggr)\bigl\{(t-\theta)\Omega\bigr\}^{s_i+\cdots+s_r}(\Omega^{s_{i-1}}H_{s_{i-1}-1})^{(-1)}L(s_j, \ldots, s_{i-2})\nonumber\\
                  \intertext{by using the equation \eqref{deqo},}
                  =&\biggl(\prod_{n=j}^{i-1}\gamma_n\biggr)\bigl\{(t-\theta)\Omega\bigr\}^{s_i+\cdots+s_r}L(s_j, \ldots, s_{i-1})\nonumber\\
                  &+\gamma_{i-1}^{(-1)}\biggl(\prod_{n=j}^{i-2}\gamma_n\biggr)\bigl\{(t-\theta)\Omega\bigr\}^{s_{i-1}+\cdots+s_r}(H_{s_{i-1}-1})^{(-1)}L(s_j, \ldots, s_{i-2}).\nonumber
                  \end{align*}

On the other hand, $(\Phi\Psi)_{ij}$ is the product of $i$-th row of $\Phi$
\[
\begin{array}{r cccccccccc l}
                      & \multicolumn{3}{c}{\overbrace{\hspace{1.3cm}}^{i-2}}& & & & & & & & \\
\ldelim( {1}{4pt}[] & 0 & \cdots & 0 & \gamma_{i-1}^{(-1)}(t-\theta)^{s_{i-1}+\cdots+s_r}H_{s_{i-1}-1}^{(-1)} & (t-\theta)^{s_i+\cdots+s_r} & 0 & \cdots & 0 & \rdelim) {1}{4pt}[]
\end{array} 
\]
and $j$-th column of $\Psi$
\begin{align}\label{cpsi}
\begin{array}{rcll}
\ldelim( {10}{4pt}[] & 0 & \rdelim) {10}{4pt}[] & \rdelim\}{3}{10pt}[$j-1$]\\
& \vdots & & \\
& 0 & & \\
& \Bigl(\prod_{n=j}^{j-1}\gamma_n\Bigr)\Omega^{s_j+\cdots+s_r} & &  \\
& \Bigl(\prod_{n=j}^{j}\gamma_n\Bigr)\Omega^{s_{j+1}+\cdots+s_r}L(s_j) & & \\
& \vdots & & \\
&\Bigl(\prod_{n=j}^{r-1}\gamma_n\Bigr)\Omega^{s_r}L(s_j,\ldots, s_{r-1}) & &  \\
& \Bigl(\prod_{n=j}^{r}\gamma_n\Bigr)L(s_j,\ldots, s_{r}) & &.
\end{array}
\end{align}
That is, 
\begin{align*}
(\Phi\Psi)_{ij}=&\gamma_{i-1}^{(-1)}(t-\theta)^{s_{i-1}+\cdots+s_r}H_{s_{i-1}-1}^{(-1)}\biggl(\prod^{i-2}_{n=j}\gamma_n\biggr)\Omega^{s_{i-1}+\cdots+s_r}L(s_{j}, \ldots, s_{i-2})\\
&+(t-\theta)^{s_{i}+\cdots+s_r}\biggl(\prod^{i-1}_{n=j}\gamma_n\biggr)\Omega^{s_{i}+\cdots+s_r}L(s_{j}, \ldots, s_{i-1}).
\end{align*}
Thus we have $(\Phi\Psi)_{ij}=\Psi^{(-1)}_{ij}$.
\end{proof}

As we will see in Lemma \ref{psie} later, the matrix $\Psi$ in \eqref{Psi} belongs to ${\rm Mat}_{r+1}({\mathbb{T}})$ then we obtain that $\Psi\in{\rm GL}_{r+1}({\mathbb{L}})$ from \eqref{omegaentire} and ${\rm det}\Psi=\Omega^{\sum_{i=1}^rd_i}\neq 0$ where $d_i=s_i+\cdots+s_r$. Therefore $\Psi$ is a rigid analytic trivialization of $\Phi$ and we call each entry of $\Psi|_{t=\theta}$ a period of $M$. Thus we have the following result by the above proposition.
\begin{thm}\label{perint}
For $\mathfrak{s}=(s_1, \ldots, s_r)\in\mathbb{N}^r$ and ${\boldsymbol \epsilon}=(\epsilon_1, \ldots, \epsilon_r) \in(\mathbb{F}_q^{\times})^r$, $\zeta_A(\mathfrak{s};{\boldsymbol \epsilon})$ are periods of the pre-$t$-motive $M$ in Definition \ref{M}.
\end{thm}
\begin{proof}

By using $\Omega|_{t=\theta}=\tilde{\pi}^{-1}$ and \eqref{atint}, we obtain 
\begin{align}\label{mplzeta}
	L(\mathfrak{s})|_{t=\theta}&=\frac{\Gamma_{s_1}\cdots\Gamma_{s_r}}{\tilde{\pi}^{s_1+\cdots+s_r}}\sum_{d_1>\cdots >d_r\geq0}\epsilon_1^{d_1}\cdots\epsilon_r^{d_r}S_{d_1}(s_1)\cdots S_{d_r}(s_r)\\
	&=\frac{\Gamma_{s_1}\cdots\Gamma_{s_r}}{\tilde{\pi}^{s_1+\cdots+s_r}}\zeta_A(\mathfrak{s}; {\boldsymbol \epsilon}).\nonumber
\end{align}
Therefore by the matrix $\Psi$ in \eqref{Psi}, $\zeta_A(\mathfrak{s}; {\boldsymbol \epsilon})$ are periods of the pre-$t$-motive $M$ in Definition \ref{M}. 
\end{proof}

\section{Linear independence of monomials of AMZVs}\label{seclinind}
In this section, we show that AMZVs form an weight-graded algebra as an application of their period interpretation in the former section. The proof is shown along the method which was invented by Chang \cite{C1}.

\subsection{ABP-criterion}
In our proof, we need to use ABP criterion (ABP stands for Anderson-Brownawell-Papanikolas), which was introduced in \cite{ABP}.

The ABP criterion is stated as the following theorem.
\begin{thm}[{\cite[Theorem 3.1.1]{ABP}}]\label{abpcri}
Fix $\Phi\in{\rm Mat}_{d}(\overline{k}[t])$ so that ${\rm det}\Phi=c(t-\theta)^s$ for some $c\in\overline{k}^{\times}$ and some $s\in\mathbb{Z}_{\geq 0}$. Suppose that there exists a vector $\psi\in{\rm Mat}_{d\times 1}(\mathbb{E})$ satisfies
\[
    \psi^{(-1)}=\Phi\psi.
\]
For every $\rho\in{\rm Mat}_{1\times d}(\overline{k})$ such that $\rho\psi(\theta)=0$, there is a $P\in{\rm Mat}_{1\times d}(\overline{k}[t])$ so that $P(\theta)=\rho$ and $P\psi=0$. 
\end{thm}

From Proposition \ref{rat}, it is a simple matter to verify that the matrices $\Phi$ in Definition \ref{M} satisfy the conditions. We thus only need to show that $\Psi \in {\rm Mat}_{r+1}(\mathbb{E})$. This is necessary in applying Theorem \ref{abpcri} to our AMZVs.
We use the following proposition which was given in \cite{ABP}.

\begin{prop}[{\cite[Proposition 3.1.3]{ABP}}]\label{abpprop1}
Suppose 
\[
	\Phi\in{\rm Mat}_{l}(\overline{k}[t]), \quad \psi\in {\rm Mat}_{l\times 1}(\mathbb{T})
\]
such that 
\[
	{\rm det}\Phi|_{t=0}\neq 0, \quad \psi^{(-1)}=\Phi\psi.
\]
Then 
\[
	\psi\in {\rm Mat}_{l\times 1}(\mathbb{E}).
\]
\end{prop}

By using this proposition, we prove the following lemma whose proof is based on \cite[Lemma 5.3.1]{C1}. 
\begin{lem}\label{psie}
Let $\Psi\in{\rm GL}_{r+1}(\overline{k}[t])$
be as in Proposition \ref{rat}. Then the following holds:
\[
	\Psi\in{\rm Mat}_{r+1}(\mathbb{E}).
\]
\end{lem}
\begin{proof}
To apply Proposition \ref{abpprop1}, we first prove that $\Psi\in{\rm Mat}_{r+1}(\mathbb{T})$.
By \eqref{omegaentire}, $\Omega\in\mathbb{T}$ thus it is sufficient to show that each 
$\Bigl(\prod_{n=i}^{j}\gamma_n\Bigr) L(\mathfrak{s};\boldsymbol{\epsilon})$ belongs to $\mathbb{T}$ where $\mathfrak{s}=(s_i, \ldots, s_j)$ and $\boldsymbol{\epsilon}=(\epsilon_i, \ldots, \epsilon_j)$ for $1\leq i\leq j\leq r$.
When $|t|_{\infty}\leq 1$, we have the following for $\sum^{m}_{n=0}a_nt^n\in\overline{k}[t]$:
\[
	||\sum^{m}_{n= 0}a_nt^n||_{\infty}=\max_{m\geq n \geq 0}\{ |a_n|_{\infty}\}\geq \max_{m\geq n \geq 0}\{ |a_nt^n|_{\infty}\}\geq |\sum^{m}_{n=0}a_nt^n|_{\infty}\geq 0.
\]
Thus if $||\sum^{m}_{n=0}a_nt^n||_{\infty}\rightarrow 0$ as $m\rightarrow \infty$, $\sum^{\infty}_{n=0}a_nt^n$ converges. Then for the following series 
\begin{align*}
	\Biggl(\prod_{n=i}^j\gamma_n\Biggr) L(\mathfrak{s};\boldsymbol{\epsilon})
				   =&\sum_{d_i>\cdots>d_j\geq 0}\gamma_i\epsilon_i^{d_i}(H_{s_i-1}\Omega^{s_i})^{(d_i)}\cdots\gamma_j\epsilon_j^{d_j}(H_{s_j-1}\Omega^{s_j})^{(d_j)}\\
				   &=\Omega^{s_i+\cdots+s_j}\sum_{d_i>\cdots>d_j\geq 0}\frac{ (\gamma_i\epsilon_i^{d_i})^{\frac{1}{q^{d_i}}}H_{s_i-1}^{(d_i)}\cdots(\gamma_j\epsilon_j^{d_j})^{\frac{1}{q^{d_j}}}H_{s_j-1}^{(d_j)}}{\bigr((t-\theta^q)\cdots(t-\theta^{q^{d_i} })\bigr)^{s_i}\cdots\bigr((t-\theta^q)\cdots(t-\theta^{q^{d_j} })\bigr)^{s_j}},		   		\end{align*}
we need to show that when $|t|_{\infty}\leq 1$,
\[
\frac{ ||(\gamma_i\epsilon_i^{d_i})^{\frac{1}{q^{d_i}}}H_{s_i-1}^{(d_i)}\cdots(\gamma_j\epsilon_j^{d_j})^{\frac{1}{q^{d_j}}}H_{s_j-1}^{(d_j)}||_{\infty} }{||\bigr((t-\theta^q)\cdots(t-\theta^{q^{d_i} })\bigr)^{s_i}\cdots\bigr((t-\theta^q)\cdots(t-\theta^{q^{d_j} })\bigr)^{s_j}||_{\infty}}\rightarrow 0
\]
 as $0\leq d_j<\cdots <d_i\rightarrow \infty$ ($d_i$ goes to infinity preserving $0\leq d_j<\cdots <d_i$).
By $\epsilon_n\in\mathbb{F}_q^{\times}$ and $\gamma_n\in\overline{\mathbb{F}_q}^{\times}$, we have $|(\gamma_n\epsilon_n^{d_n})^{\frac{1}{q^{d_n}}}|_{\infty}=1$ and thus
\begin{align}\label{aa}
 ||(\gamma_n\epsilon_n^{d_n})^{\frac{1}{q^{d_n}}}H_{s_n-1}^{(d_n)}||_{\infty}=||H_{s_n-1}^{(d_n)}||_{\infty}.
\end{align}
Moreover, by $|t|_{\infty}<|\theta|_{\infty}$ we have $|t-\theta^{q^n}|_{\infty}=|\theta^{q^n}|_{\infty}$ and then 
\begin{align*}
	||\bigr((t-\theta^q)\cdots(t-\theta^{q^{d_i} })\bigr)^{s_i}\cdots\bigr((t-\theta^q)\cdots(t-\theta^{q^{d_j} })\bigr)^{s_j}||_{\infty}=|(\theta^q\cdots\theta^{q^{d_i}})^{s_i}\cdots(\theta^{q}\cdots\theta^{q^{d_j}})^{s_j}|_{\infty}.
\end{align*}
For each $i\leq n\leq j$, we have 
\begin{align}\label{ab}
	\frac{1}{|\theta^{q+\cdots+q^{d_n}}|^{s_n}_{\infty}}=\frac{1}{q^{s_n(q^{d_n}+q^{d_n-1}+\cdots+q)}}=\frac{1}{ q^{qs_n(q^{d_n-1}+q^{d_n-2}+\cdots+1)} }=\frac{q^{qs_n/(q-1)}}{(q^{qs_n/(q-1)})^{q^{d_n}}}.
\end{align}
By using $\eqref{degh}$, $\eqref{aa}$ and $\eqref{ab}$ we have 
\begin{align*}
   &\frac{ ||(\gamma_i\epsilon_i^{d_i})^{\frac{1}{q^{d_i}}}H_{s_i-1}^{(d_i)}\cdots(\gamma_j\epsilon_j^{d_j})^{\frac{1}{q^{d_j}}}H_{s_j-1}^{(d_j)}||_{\infty} }{||\bigr((t-\theta^q)\cdots(t-\theta^{q^{d_i} })\bigr)^{s_i}\cdots\bigr((t-\theta^q)\cdots(t-\theta^{q^{d_j} })\bigr)^{s_j}||_{\infty}}\\
    &=\Bigl(||H_{s_i-1}||^{q^{d_i}}_{\infty}/|\theta^{q+\cdots+q^{d_i}}|^{s_i}_{\infty}\Bigr)\cdots\Bigl(||H_{s_j-1}||^{q^{d_j}}_{\infty}/|\theta^{q+\cdots+q^{d_j}}|^{s_j}_{\infty}\Bigr)\\
    &\leq \biggl(\Bigl(q^{\frac{s_i-1}{q-1}q}\Bigr)^{q^{d_i}}q^{\frac{qs_i}{q-1}}/\Bigl(q^{\frac{qs_i}{q-1}} \Bigr)^{q^{d_i}}  \biggr)\cdots\biggl(\Bigl(q^{\frac{s_j-1}{q-1}q}\Bigr)^{q^{d_j}}q^{\frac{qs_j}{q-1}}/\Bigl(q^{\frac{qs_j}{q-1}} \Bigr)^{q^{d_j}}  \biggr)\\
    &=q^{\frac{q}{q-1}(s_i+\cdots+s_j)}(q^{-s_i-\frac{1}{q-1}})^{q^{d_i}}\cdots (q^{-s_j-\frac{1}{q-1}})^{q^{d_j}}. 
\end{align*}
It is clear that $(q^{-s_i-\frac{1}{q-1}})^{q^{d_i}}\cdots (q^{-s_j-\frac{1}{q-1}})^{q^{d_j}}\rightarrow 0$ as $0\leq d_j<\cdots <d_i\rightarrow \infty$.

Thus we obtain the desired conclusion and therefore
\begin{align}\label{psitate}
	\Psi\in{\rm Mat}_{r+1}(\mathbb{T}).
\end{align}
 
Now we complete the proof. Let $\psi_{i}\in{\rm Mat}_{r+1\times 1}(\mathbb{T})\ (1\leq i\leq r+1)$ be the column of the matrix $\Psi$. Then by using Proposition \ref{rat}, we can show that $\Phi\in{\rm Mat}_{r+1}(\overline{k}[t])$ in Definition \ref{M} and $\psi_i$ satisfy $\psi_i^{(-1)}=\Phi\psi_i$ for each $i$. Furthermore, $\Phi$ satisfies ${\rm det}\Phi|_{t=0}=(-\theta)^{\sum_{i=1}^rd_i}\neq 0$ where $d_i=s_i+\cdots+s_r$. 
Thus we may apply Proposition \ref{abpprop1} to $\Phi$ and $\psi_i$ and therefore we obtain
\[ 
 	\Psi\in{\rm Mat}_{r+1}(\mathbb{E}).
\]
\end{proof}

\subsection{MZ property for AMZVs}
First we verify that AMZVs satisfy the following lemma which is alternating analogue of MZ property in \cite{C1}.
 \begin{lem}\label{amzproperty}
 For a given AMZV $\zeta_{A}(\mathfrak{s};\boldsymbol{\epsilon})$ with ${\rm wt}(\mathfrak{s})=w$ and ${\rm dep}(\mathfrak{s})=r$, there exists $\Phi\in{\rm Mat}_{r+1}(\overline{k}[t])$ and $\psi\in{\rm Mat}_{(r+1)\times 1}(\mathbb{E})$  with $r\geq 1$ such that:
 \begin{enumerate}
     \item [(i)] $\psi^{(-1)}=\Phi\psi$ and $\Phi$ satisfies the condition of Theorem \ref{abpcri}$;$
     \item[(ii)] the last column of $\Phi$ is of the form $(0, \ldots, 0, 1)^{{\rm tr}};$
     \item[(iii)] for some $a\in\overline{\mathbb{F}}_q^{\times}$ and $b\in k^{\times}$, $\psi(\theta)$ is of the form with specific first and last entries
                     \[
                         \psi(\theta)=\Biggl(\frac{1}{\tilde{\pi}^w}, \ldots, a\frac{b\zeta_A(\mathfrak{s};\boldsymbol{\epsilon})}{\tilde{\pi}^w}\Biggr)^{\rm tr};
                     \]
     \item[(iv)] for any $N\in\mathbb{N}$ and some $c\in \mathbb{F}_q^{\times}$, $\psi(\theta^{q^N})$ is of the form  
                      \[
                         \psi(\theta^{q^N})=\Biggl(0, \ldots, 0,  ac^{N}\Bigl(\frac{b\zeta_A(\mathfrak{s};\boldsymbol{\epsilon})}{\tilde{\pi}^w}\Bigr)^{q^{N}}\Biggr)^{\rm tr}.
                     \]
 \end{enumerate}
 \end{lem}
 \begin{proof}
We may take matrices $\Phi\in {\rm Mat}_{r+1}(\overline{k}[t])$ and $\Psi\in {\rm Mat}_{r+1}(\mathbb{E})$ ($r\geq 1$) for each $\zeta_{A}(\mathfrak{s};\boldsymbol{\epsilon})$ with ${\rm wt}(\mathfrak{s})=w$ and ${\rm dep}(\mathfrak{s})=r$ as in Definition \ref{M} and \eqref{Psi} respectively.
Let us denote $\psi_i\in{\rm Mat}_{(r+1)\times 1}(\mathbb{E})$ ($1\leq i \leq r+1$) the $i$-th column of $\Psi$.  
From Definition \ref{M}, Proposition \ref{rat} and Lemma \ref{psie}, it is evident that $\Phi$ and each $\psi_i$ satisfy (i)-(iii). Thus it is enough to check the condition (iv) for each $\psi_i$.  
For each $i$ ($1\leq i \leq r$), we put  
\begin{align*} 
L(\mathfrak{s})_{i}&:=\sum_{d_i>\cdots>d_r\geq 0}\epsilon_i^{d_i}\cdots\epsilon_r^{d_r}(\Omega^{s_i}H_{s_i-1})^{(d_i)}\cdots(\Omega^{s_r}H_{s_r-1})^{(d_r)},\\
L(\mathfrak{s})_{i, {\geq N}}&:=\sum_{d_i>\cdots>d_r\geq N}\epsilon_i^{d_i}\cdots\epsilon_r^{d_r}(\Omega^{s_i}H_{s_i-1})^{(d_i)}\cdots(\Omega^{s_r}H_{s_r-1})^{(d_r)},\\
L(\mathfrak{s})_{i, {< N}}&:=\sum_{\substack{d_i>\cdots>d_r\\d_r<N}}\epsilon_i^{d_i}\cdots\epsilon_r^{d_r}(\Omega^{s_i}H_{s_i-1})^{(d_i)}\cdots(\Omega^{s_r}H_{s_r-1})^{(d_r)}.
\end{align*}
It is obvious that 
$L(\mathfrak{s})_i=L(\mathfrak{s})_{i, {\geq N}}+L(\mathfrak{s})_{i, {< N}}.$
By the definition, $\Omega$ has simple zero at $t=\theta^{q^N}$ and its $(s_i+\cdots+s_r)$-th power has $s_i+\cdots+s_r$ order of zero at $t=\theta^{q^N}$ while the denominator of each term in $L(\mathfrak{s})_{i, < N} $ has at most $s_i+\cdots+s_{r-1}$ order of zero at $t=\theta^{q^N}$. Then we have $L(\mathfrak{s})_{i, < N}|_{t=\theta^{q^N}}=0$ and thus $L(\mathfrak{s})_i|_{t=\theta^{q^N}}=L(\mathfrak{s})_{i, \geq N}|_{t=\theta^{q^N}}$. Here we note that $L(\mathfrak{s})_{i, \geq N}|_{t=\theta^{q^N}}$ converges because of both $\Omega^{s_i+\cdots+s_r}$ and denominator of each term has $s_i+\cdots+s_r$ order of zero at $t=\theta^{q^N}$. We have the following equalities:
\begin{align*}
&L(\mathfrak{s})_i|_{t=\theta^{q^N}}=L(\mathfrak{s})_{i, \geq N}|_{t=\theta^{q^N}}=\sum_{d_i>\cdots>d_r\geq N}\epsilon_i^{d_i}\cdots\epsilon_r^{d_r}\bigl\{(\Omega^{s_i}H_{s_i-1})^{(d_i)}\cdots(\Omega^{s_r}H_{s_r-1})^{(d_r)}\bigr\}|_{t=\theta^{q^N}}\\
&=\sum_{d_i>\cdots>d_r\geq 0}\epsilon_i^{d_i+N}\cdots\epsilon_r^{d_r+N}\bigl\{(\Omega^{s_i}H_{s_i-1})^{(d_i+N)}\cdots(\Omega^{s_r}H_{s_r-1})^{(d_r+N)}\bigr\}|_{t=\theta^{q^N}}\\
&=\epsilon_i^{N}\cdots\epsilon_r^{N}\sum_{d_i>\cdots>d_r\geq 0}\epsilon_i^{d_i}\cdots\epsilon_r^{d_r}\bigl\{(\Omega^{s_i}H_{s_i-1})^{(d_i+N)}\cdots(\Omega^{s_r}H_{s_r-1})^{(d_r+N)}\bigr\}|_{t=\theta^{q^N}}\\
&=\epsilon_i^{N}\cdots\epsilon_r^{N}\Bigl(\sum_{d_i>\cdots>d_r\geq 0}\epsilon_i^{d_i}\cdots\epsilon_r^{d_r}\bigl\{(\Omega^{s_i}H_{s_i-1})^{(d_i)}\cdots(\Omega^{s_r}H_{s_r-1})^{(d_r)}\bigr\}|_{t=\theta}\Bigr)^{q^N}.
\end{align*} 
Let $(\Omega^{s_i}H_{s_i-1})^{(d_i)}\cdots(\Omega^{s_r}H_{s_r-1})^{(d_r)}=\sum_{n\geq 0}a_nt^n\in\mathbb{C}_{\infty}[[t]]$, then the last equality in the above is shown by the following:
\[
	\Bigl\{\Bigl(\sum_{n\geq 0}a_nt^n\Bigr)^{(N)}\Bigr\}|_{t=\theta^{q^N}}=\sum_{n\geq 0}a_n^{q^N}\theta^{nq^N}=\Bigl\{\Bigl(\sum_{n\geq 0}a_nt^n\Bigr)|_{t=\theta}\Bigr\}^{q^N}.
\]
Thus we have
\[
	L(\mathfrak{s})_i|_{t=\theta^{q^N}}=\epsilon_i^{N}\cdots\epsilon_r^{N}L(\mathfrak{s})_i^{q^N}.
\]
Therefore by using $\Omega|_{t=\theta^{q^N}}=0$, \eqref{cpsi} and \eqref{mplzeta}, indeed we obtain  
 \[
                         \psi_i(\theta^{q^N})=\Biggl(0, \ldots, 0,  ac^{N}\Bigl(\frac{b\zeta_A(\mathfrak{s}_i;\boldsymbol{\epsilon}_i)}{\tilde{\pi}^{w_i}}\Bigr)^{q^{N}}\Biggr)^{\rm tr}
  \]
where $\mathfrak{s}_i=(s_i, \ldots, s_r)$, $\boldsymbol{ \epsilon}_i=(\epsilon_i, \ldots, \epsilon_r)$, $w_i={\rm wt}(\mathfrak{s}_i)$, $a=\prod^{r}_{j=i}\gamma_{j}$, $b=\Gamma_{s_i}\cdots\Gamma_{s_r}$ and $c=\prod^{r}_{j=i}\epsilon^{j}$ for each $i$.
\end{proof}
Next we show that monomials of AMZVs also satisfy Lemma \ref{amzproperty} by using the same method in proof of \cite[Proposition 3.4.4]{C1}.
\begin{prop}\label{monoamzprop}
We let $\zeta_A(\mathfrak{s}_1;\boldsymbol{\epsilon}_1), \ldots, \zeta_A(\mathfrak{s}_n;\boldsymbol{\epsilon}_n)$ AMZVs with weights $w_1,$ $\ldots, w_n$ respectively and let $m_1, \ldots, m_n\in\mathbb{Z}_{\geq 0}$. Then there exist matrices $\Phi\in{\rm Mat}_d(\overline{k}[t])$ and $\psi\in{\rm Mat}_{d\times 1}(\mathbb{E})$ with $d\geq 2$ so that  $(\Phi, \psi, \zeta_A(\mathfrak{s}_1;\boldsymbol{\epsilon}_1)^{m_1} \cdots \zeta_A(\mathfrak{s}_n;\boldsymbol{\epsilon}_n)^{m_n})$ satisfies (i)-(iv) in Lemma \ref{amzproperty}. 
\end{prop}
\begin{proof}
We take triple $(\Phi_i, \psi_i, \zeta_A({\mathfrak s}_i;{\boldsymbol \epsilon}_i))$ which satisfies Lemma \ref{amzproperty} for each $i$. Then we consider the Kronecker product $\otimes$ (See \cite[Chapter 8]{Sc}) of $\Phi_i$ and $\psi_i$ respectively as followings:
\[
	\Phi:=\Phi_1^{\otimes m_1}\otimes\cdots\otimes\Phi_n^{\otimes m_n},\quad \psi:=\psi_1^{\otimes m_1}\otimes\cdots\otimes\psi_n^{m_n}.
\]
By our assumption, $(\Phi_i, \psi_i, \zeta_A({\mathfrak s}_i;{\boldsymbol \epsilon}_i))$ satisfy Lemma \ref{amzproperty} and thus by using the property of Kronecker product which involves matrix multiplication (cf. \cite[Theorem 7.7]{Sc}),  the triple $(\Phi, \psi, \zeta_A(\mathfrak{s}_1;\boldsymbol{\epsilon}_1)^{m_1} \cdots \zeta_A(\mathfrak{s}_n;\boldsymbol{\epsilon}_n)^{m_n})$ does so.  
\end{proof}

\begin{defn}
Let $\zeta_A({\mathfrak s}_1;{\boldsymbol \epsilon}_1 ), \ldots, \zeta_A({\mathfrak s}_n;{\boldsymbol \epsilon}_n )$ be AMZVs of ${\rm wt}({\mathfrak s}_i)=w_i$ $(i=1, \ldots, n)$. For $m_1, \ldots, m_n\in\mathbb{Z}_{\geq0}$ not all zero, we define the total weight of the monomial $\zeta_A({\mathfrak s}_1;{\boldsymbol \epsilon}_1 )^{m_1} \cdots \zeta_A({\mathfrak s}_n;{\boldsymbol \epsilon}_n )^{m_n}$ as 
\[
	\sum^n_{i=1}m_iw_i.
\]
For $w\in\mathbb{N}$, we denote $AZ_{w}$ the set of monomials of AMZVs with total weight $w$.  
\end{defn}
We note that  $AZ_{w}$ is finite set.

Now we prove the linear independence of monomials of AMZVs.
\begin{thm}\label{linindamz}
Let $w_1, \ldots, w_l\in\mathbb{N}$ be distinct. We suppose that $V_i$ is a $k$-linearly independent subset of $AZ_{w_i}$ for $i=1, \ldots, l$. Then the following union
\[
    \{ 1 \}\bigcup_{i=1}^{l}V_{i}
\]
is a linearly independent set over $\overline{k}$, that is, there are no nontrivial $\overline{k}$-linear relation among elements of $\{ 1 \}\bigcup_{i=1}^{l}V_{i}$.
\end{thm}
\begin{proof}
We may assume that $w_l>\cdots>w_1$ without loss of generality. For each $i=1, \ldots, l$,  
$AZ_{w_i}$ is a finite set by definition and thus its subset $V_{i}$ is also finite. Let $V_{i}$ consist of $\{ Z_{i1}, \ldots, Z_{im_i}  \}$ where $Z_{ij}\in AZ_{w_i} (j=1, \ldots, m_i)$ are the same total weight $w_i$. 
The proof is by induction on weight $w_l$.

We require on the contrary that $\{1\}\bigcup_{i=1}^lV_i$ is $\overline{k}$-linearly dependent set. 
Then we may also assume that there are nontrivial $\overline{k}$-linear relations 
\[ 
a_0\cdot 1+a_{11}Z_{11}+\cdots+a_{1m_1}Z_{1m_1}+\cdots+a_{l1}Z_{l1}+\cdots+a_{lm_l}Z_{lm_l}=0,
\]
we may take $a_0, a_{11}, \ldots, a_{lm_l}\in\overline{k}$ with $a_{li}\neq 0$ for some $i=1, \ldots, m_l$.

We proceed our proof by assuming the existence of nontrivial $\overline{k}$-linear relations between $V_l$ and $\{1\}\bigcup_{i=1}^{l-1}V_i$. 

For $1\leq i\leq l$ and $1\leq j\leq m_l$, by combining Proposition \ref{rat} with Proposition \ref{monoamzprop}, there exist the matrices 
\begin{align}\label{mats}
    \Phi_{ij}\in{\rm Mat}_{d_{ij}}(\overline{k}[t])\quad {\rm and}\quad \psi_{ij}\in{\rm Mat}_{d_{ij}\times 1}({\mathbb E})
\end{align}
so that $d_{ij}\geq 2$ and each $(\Phi_{ij}, \psi_{ij}, Z_{ij})$ satisfy Lemma \ref{amzproperty}.

For the matrix $\Phi_{ij}$ and the column vector $\psi_{ij}$, we define the following block diagonal matrix and the column vector 
\[
    \tilde{\Phi}:=\bigoplus_{i=1}^{l}\biggl( \bigoplus^{m_i}_{j=1}(t-\theta)^{w_l-w_i}\Phi_{ij} \biggr)\quad \text{and}\quad \tilde{\psi}:=\bigoplus^{l}_{i=1}\biggl( \bigoplus^{m_i}_{j=1}\Omega^{w_l-w_i}\psi_{ij} \biggr).
\]
In this proof, we define the direct sum of any column vectors ${\bf v}_1, \ldots, {\bf v_m}$ whose entries belong to $\mathbb{C}_{\infty}((t))$ by $\bigoplus^{m}_{i=1}{\bf v}_i:=({\bf v}_1^{\rm tr}, \ldots, {\bf v}_m^{\rm tr})^{\rm tr}$.  

By the requirement, $\{ 1 \}\bigcup_{i=1}^l V_i$ is a linearly dependent over $\overline{k}$. Thus there exists a nonzero vector
\[
	\rho=({\bf v}_{11}, \ldots, {\bf v}_{1m_1}, \ldots, {\bf v}_{l1}, \ldots, {\bf v}_{lm_l}) 
\] 
 such that 
\begin{align*}
 \rho\cdot(\tilde{\psi}|_{t=\theta})&=\rho\cdot\bigoplus_{i=1}^{l}\bigoplus^{m_i}_{j=1}\Biggl(\frac{1}{\tilde{\pi}^{w_l}}, \ldots, a\frac{bZ_{ij}}{\tilde{\pi}^{w_l}}\Biggr)^{\rm tr}\\
 &=\frac{1}{\tilde{\pi}^{w_l}}({\bf v}_{11}, \ldots, {\bf v}_{1m_1}, \ldots, {\bf v}_{l1}, \ldots, {\bf v}_{lm_l}) \bigoplus^{l}_{i=1}\bigoplus^{m_i}_{j=1}\Biggl(1, \ldots, abZ_{ij}\Biggr)^{\rm tr}=0, 
\end{align*}
 where ${\bf v}_{ij}\in {\rm Mat}_{1\times d_{ij}}(\overline{k})$ for $1\leq i\leq l$ and $1\leq j\leq m_i$. Then we have nontrivial $\overline{k}$-linear relation 
 \[
 	({\bf v}_{11}, \ldots, {\bf v}_{1m_1}, \ldots, {\bf v}_{l1}, \ldots, {\bf v}_{lm_l}) \bigoplus^{l}_{i=1}\bigoplus^{m_i}_{j=1}\Biggl(1, \ldots, abZ_{ij}\Biggr)^{\rm tr}=0.
 \]
 In the beginning of this proof, we assumed that there exists nontrivial $\overline{k}$-linear relations between $V_l$ and $\{1 \}\bigcup_{i=1}^{l-1}V_i$ and then for some $1\leq s\leq m_l$, the last entry of ${\bf v}_{ls}$ is nonzero. Since the last entry in ${\bf v}_{li}$ is coefficient of $abZ_{li}$ for $1\leq i\leq m_l$ in the above relation. By using Theorem \ref{abpcri}, we have ${\bf F}:=({\bf f}_{11}, \ldots, {\bf f}_{1m_1}, \ldots, {\bf f}_{l1}, \ldots, {\bf f}_{lm_l})$ where ${\bf f}_{ij}=(f_{i1}, \ldots, f_{id_{ij}})\in{\rm Mat}_{1\times d_{ij}}(\overline{k}[t])$ for $1\leq i\leq l$, $1\leq j\leq m_i$ and it satisfies 
 \[
 	{\bf F}\tilde{\psi}=0\quad \text{and}\quad {\bf F}|_{t=\theta}=\rho. 
 \] 
 The last entry of ${\bf f}_{ls}$ is a nontrivial polynomial because the last entry of ${\bf v}_{ls}$ is not zero. We choose a sufficiently large $N\in\mathbb{Z}$ so that ${\bf f}_{ls}|_{t=\theta^{q^N}}\neq 0$. 
We rewrite the equation $({\bf F}\tilde{\psi})|_{t=\theta^{q^N}}=0$ by using $\Omega|_{t=\theta^{q^N}}=0$, Lemma \ref{amzproperty} (iv) and the definition of $\tilde{\psi}$ as follows:
\begin{align*}
	&({\bf F}\tilde{\psi})|_{t=\theta^{q^N}}\\
	&=({\bf f}_{11}, \ldots, {\bf f}_{1m_1}, \ldots, {\bf f}_{l1}, \ldots, {\bf f}_{lm_l})|_{t=\theta^{q^N}}\bigoplus^{l}_{i=1}\bigoplus^{m_i}_{j=1}\Omega^{w_l-w_i}|_{t=\theta^{q^N}}\Biggl(0, \ldots, 0, a_jc_j^{N}\Bigl(\frac{b_jZ_{ij}}{\tilde{\pi}^{w_i}}\Bigr)^{q^N}\Biggr)^{\rm tr}\\
	&=({\bf f}_{11}, \ldots, {\bf f}_{1m_1}, \ldots, {\bf f}_{l1}, \ldots, {\bf f}_{lm_l})|_{t=\theta^{q^N}}\bigoplus^{m_l}_{j=1}\Biggl(0, \ldots, 0, a_jc_j^{N}\Bigl(\frac{b_jZ_{lj}}{\tilde{\pi}^{w_l}}\Bigr)^{q^N}\Biggr)^{\rm tr}\\
	&=(f_{11}, \ldots, f_{1d_{11}}, \ldots, f_{l1}, \ldots, f_{ld_{lm_l}})|_{t=\theta^{q^N}}\bigoplus^{m_l}_{j=1}\Biggl(0, \ldots, 0, a_jc_j^{N}\Bigl(\frac{b_jZ_{lj}}{\tilde{\pi}^{w_l}}\Bigr)^{q^N}\Biggr)^{\rm tr}\\
	&=\sum_{j=1}^{m_l}(f_{ld_{lj}}|_{t=\theta^{q^N}})a_jc_j^{N}\Bigl(\frac{b_jZ_{lj}}{\tilde{\pi}^{w_l}}\Bigr)^{q^N}=0.
\end{align*}
Thus we obtain the following nontrivial $\overline{k}$-linear relation with some $f_{ld_{ls}}\neq0$:
\[
	\sum_{j=1}^{m_l}(f_{ld_{lj}}|_{t=\theta^{q^N}})a_jc_j^{N}\Bigl(b_jZ_{lj}\Bigr)^{q^N}=0.
\] 
 Therefore by taking $q^N$th root of the above $\overline{k}$-linear relation, we get a nontrivial relation for $\{  Z_{l1}, \ldots, Z_{lm_l}  \}$ as follows. 
\[
	\sum_{j=1}^{m_l}\Bigl\{(f_{ld_{lj}}|_{t=\theta^{q^N}})a_jc_j^{N}\Bigr\}^{\frac{1}{q^N}}b_jZ_{lj}=0.
\] 
This shows that $V_l$ is a $\overline{k}$-linearly dependent set.  
Therefore by using Lemma \ref{laslem}, we can show that $V_l$ is a $k$-linearly dependent subset.
However, it contradicts the condition saying that $V_l$ is the $k$-linearly independent set. Therefore we obtain the claim. 
\end{proof}

\begin{lem}\label{laslem}
Let $V$ be a finite $\overline{k}$-linearly dependent subset of $AZ_{w}$. Then $V$ is a finite $k$-linearly dependent subset of $AZ_{w}$. 
\end{lem}
\begin{proof}
We put $V=\{Z_{1}, \ldots, Z_{m}\}$. Without loss of generality, we may assume that $m\geq 2$ by Theorem \ref{nonvan} and 
\begin{align}\label{dim}
	{\rm dim}_{\overline{k}}{\rm Span}_{ \overline{k} }\{ V \}=m-1.
\end{align}
Again we may assume that $Z_{1}\in{\rm Span}_{\overline{k}}\{  Z_{2}, \ldots, Z_{m}  \}$, then by the assumption \eqref{dim}, $\{ Z_{2}, \ldots, Z_{m} \}  $ is a linearly independent set over $\overline{k}$. 
As in the proof of our previous theorem, we take the matrix $\Phi_{j}$, and the column vector $\psi_{j}$ $(1\leq j\leq m)$ so that the triple $(\Phi_{j}, \psi_{j}, Z_{j})$ satisfying Lemma \ref{amzproperty} (i)-(iv), we define block diagonal matrix $\Phi$ and column vector $\psi$ as follows. 
\[
	\Phi:=\bigoplus_{j=1}^{m}\Phi_{j}\quad \text{and}\quad \psi:=\bigoplus_{j=1}^{m}\psi_{j}. 
\]
In the above, we again define the direct sum of column vectors ${\bf v}_1, \ldots, {\bf v_m}$ whose entries belong to $\mathbb{C}_{\infty}((t))$ by $\bigoplus^{m}_{i=1}{\bf v}_i:=({\bf v}_1^{\rm tr}, \ldots, {\bf v}_m^{\rm tr})^{\rm tr}$. 
By definition, each $(\Phi_{j}, \psi_{j}, Z_{j})$ satisfy Lemma \ref{amzproperty} (i)-(iv) and then we have 
\begin{align}\label{psiw1}
                         \psi|_{t=\theta}=\bigoplus_{j=1}^{m}\Biggl(\frac{1}{\tilde{\pi}^{w}}, \ldots, a_j\frac{b_jZ_{j}}{\tilde{\pi}^{w}}\Biggr)^{\rm tr}
\end{align}
for some $a_j\in\overline{\mathbb{F}}_q^{\times}$, $b_j\in k^{\times}$ and  
 \begin{align}\label{psinw1}
                         \psi|_{t=\theta^{q^N}}=\bigoplus_{j=1}^{m}\Biggl(0, \ldots, 0,  a_jc_j^{N}\Bigl(\frac{b_jZ_{j}}{\tilde{\pi}^{w}}\Bigr)^{q^{N}}\Biggr)^{\rm tr}.
\end{align}
for some $c_j\in \mathbb{F}_q^{\times}$, $N\in\mathbb{N}$.
By using Theorem \ref{abpcri}, there exist row vectors ${\bf f}_j=({\bf f}_{j, 1}, \ldots, {\bf f}_{j, d_{j}})\in{\rm Mat}_{1\times d_{j}}(\overline{k}[t])$ ($j = 1, \ldots, m$) so that if we put ${\bf F}=({\bf f}_1, \ldots, {\bf f}_{m})$, then we have
\[
	{\bf F}\psi=0, \\ \text{ ${\bf f}_{1, d_{1}}|_{t=\theta}=1$  and  ${\bf f}_{j, i}|_{t=\theta}=0$ for $1\leq i <d_{j}$} 
\] 
since ${\bf f}_{1, d_{1}}|_{t=\theta}$ is coefficient of $a_1b_1Z_{1}/(\tilde{\pi}^{w})$ in $({\bf F}\psi)|_{t=\theta}=0$ and by the assumption \eqref{dim}, $Z_{1}$ is expressed by nontrivial $\overline{k}$-linear combinations of $Z_{2}, \ldots, Z_{m}$.   
We write ${\bf F}':=(1/{\bf f}_{1, d_{1}}){\bf F}$ and $d:=\sum_{j=1}^{m}d_{j}$. Note that the vector ${\bf F}'$ is of the form 
\[
	{\bf F}'=({\bf f}'_{1,1}, \ldots, {\bf f}'_{1, d_{1}}, \ldots, {\bf f}'_{m,1}, \ldots, {\bf f}'_{m, d_{m}} )\in{\rm Mat}_{1\times d}(\overline{k}(t))
\]
where ${\bf f}'_{1, d_{1}}=1$. We have the following from ${\bf F}\psi=0$ and ${\bf f}_{j, i}|_{t=\theta}=0$ for $1\leq i <d_{j}$: 
\begin{align}\label{qpsi0}
	{\bf F}'\psi=0\ \text{and}\ {\bf f}'_{j, i}|_{t=\theta}=0\ \text{ for all $1\leq i< d_{j}$}.
\end{align}

By using Lemma \ref{amzproperty}, we obtain ${\bf F}'^{(-1)}\Phi\psi=({\bf F}'\psi)^{(-1)}=0$ and thus
\begin{align}\label{eqnQ}
	{\bf F}'\psi-{\bf F}'^{(-1)}\Phi\psi=({\bf F}'-{\bf F}'^{(-1)}\Phi)\psi=0.
\end{align}
The last column of the matrix $\Phi_{j}$ is $(0, \ldots, 0, 1)^{\rm tr}$ for each $j$ and consequently, the $d_{1}$-th entry of row vector ${\bf F'}-{\bf F'}^{(-1)}\Phi$ is zero since the $d_{1}$-th entry of the row vectors ${\bf F'}$ and ${\bf F'}^{(-1)}\Phi$ are 1. 
The $\sum_{i=1}^{j}d_{i}$-th column of $\Phi$ is 
\begin{align*}
\begin{array}{rcll}
\ldelim( {7}{4pt}[] & 0 & \rdelim) {7}{4pt}[] & \rdelim\}{4}{10pt}[$\sum_{i=1}^{j}d_{i}$] \\
& \vdots & & \\
& 0 & & \\
& 1 & &  \\
& 0 & & \\
& \vdots & & \\
&0 & &.
\end{array}
\end{align*}
Then the $\sum_{i=1}^{j}d_{i}$-th entry of the row vector ${\bf F}'^{(-1)}\Phi$ is ${\bf f'}_{j, d_{j}}^{(-1)}$. Thus the $\sum_{i=1}^{j}d_{i}$-th entry of ${\bf F}'-{\bf F}'^{(-1)}\Phi$ is written as follows.
\[ 
	{\bf f'}_{j, d_{j}}-{\bf f'}_{j, d_{j}}^{(-1)}\quad \text{ for $j=1, \ldots, m$.} 
\]
We have the equation ${\bf f'}_{j, d_{j}}-{\bf f'}_{j, d_{j}}^{(-1)}=0$ for $j=2, \ldots, m$. 
Indeed, if there exist some $2\leq j\leq m$ so that ${\bf f'}_{j, d_{j}}-{\bf f'}_{j, d_{j}}^{(-1)}\neq 0$, we can derive the contradiction in the following way: 

Let us take sufficiently large $N\in\mathbb{N}$ so that $({\bf f'}_{j, d_{j}}-{\bf f'}_{j, d_{j}}^{(-1)})|_{t=\theta^{q^N}}\neq 0$ and all entries of $({\bf F}'-{\bf F'}^{(-1)}\Phi)$ are regular at $t=\theta^{q^{N}} $. By using \eqref{psinw1} and substituting $t=\theta^{q^N}$ in \eqref{eqnQ}, we obtain
\begin{align*}
	\Bigl\{({\bf F}'-{\bf F}'^{(-1)}\Phi)\psi\Bigl\}|_{t=\theta^{q^N}}&=\Bigl\{({\bf F}'-{\bf F}'^{(-1)}\Phi)\Bigr\}|_{t=\theta^{q^N}}\bigoplus^{m}_{j=1}\Biggl(0, \ldots, 0, a_jc_j^{N}\Bigl(\frac{b_jZ_{j}}{\tilde{\pi}^{w}}\Bigr)^{q^N}\Biggr)^{\rm tr}\\
	&=\sum_{j=1}^{m}({\bf f'}_{j, d_{j}}-{\bf f'}_{j, d_{j}}^{(-1)})|_{t=\theta^{q^N}}a_jc_j^{N}\Bigl(\frac{b_jZ_{j}}{\tilde{\pi}^{w}}\Bigr)^{q^N}\\
	&=\frac{1}{\tilde{\pi}^{w}}\sum_{j=1}^{m}({\bf f'}_{j, d_{j}}-{\bf f'}_{j, d_{j}}^{(-1)})|_{t=\theta^{q^N}}a_jc_j^{N}\Bigl(b_jZ_{j}\Bigr)^{q^N}=0.
\end{align*}
Thus combining with the ${\bf f'}_{1, d_{1}}-{\bf f'}_{1, d_{1}}^{(-1)}=1-1=0$, we obtain a nontrivial $\overline{k}$-linear relations among $Z_{2}^{q^N}, \ldots, Z_{m}^{q^N}$ as follows: 
 \[
 	\sum_{j=2}^{m}({\bf f'}_{j, d_{j}}-{\bf f'}_{j, d_{j}}^{(-1)})|_{t=\theta^{q^N}}a_jc_j^{N}\Bigl(b_jZ_{j}\Bigr)^{q^N}=0.
 \]
 Then by taking $q^{N}$th root of the relation, we obtain the following nontrivial $\overline{k}$-linear relation among $Z_{2}, \ldots, Z_{m}$:
 \[
 	\sum_{j=2}^{m}\Bigl\{({\bf f'}_{j, d_{j}}-{\bf f'}_{j, d_{j}}^{(-1)})|_{t=\theta^{q^N}}a_jc_j^{N}\Bigr\}^{\frac{1}{q^N}}b_jZ_{j}=0.
 \]
 This contradicts our assumption saying that $\{ Z_{2}, \ldots, Z_{m} \}$ is a $\overline{k}$-linearly independent set. 

Therefore we get ${\bf f'}_{j, d_{j}}-{\bf f'}_{j, d_{j}}^{(-1)}=0$ for $j=2, \ldots, m$ and this equation shows the following:
 \begin{align}\label{coefrat}
 {\bf f'}_{j, d_{j}}\in k \quad (j=2, \ldots, m).
 \end{align} 
 By substituting $t=\theta$ in the equation ${\bf F'}\psi=0$ and using \eqref{psiw1}, \eqref{qpsi0}, we obtain the following equalities:
 \begin{align*}
 	({\bf F'}\psi)|_{t=\theta}&=(0, \ldots, 0, {\bf f}'_{1, d_{1}}, \ldots, 0, \ldots, 0, {\bf f}'_{m, d_{m}})|_{t=\theta}\bigoplus_{j=1}^{m}\Biggl(\frac{1}{\tilde{\pi}^{w}}, \ldots, a_j\frac{b_jZ_{j}}{\tilde{\pi}^{w}}\Biggr)^{\rm tr}\\
 					      &=\sum_{j=1}^{m}({\bf f'}_{j, d_{j}}|_{t=\theta})a_j\frac{b_jZ_{j}}{\tilde{\pi}^{w}}=\frac{1}{\tilde{\pi}^{w}}\sum_{j=1}^{m}({\bf f'}_{j, d_{j}}|_{t=\theta})a_jb_jZ_{j}=0.
 \end{align*}
By ${\bf f}'_{1, d_{1}}=1$ and \eqref{coefrat}, we have the following nontrivial $k$-linear relation among $Z_{1}, \ldots, Z_{m}$:
\[
	\sum_{j=1}^{m}({\bf f'}_{j, d_{j}}|_{t=\theta})a_jb_jZ_{j}=0.
\]
Therefore we obtain the claim. 
\end{proof}
The above theorem provides an alternating analogue of Theorem 2.2.1 in \cite{C1}. Finally, we define the following notations and state the result.  
\begin{nota}
We denote $\overline{\mathcal{ AZ }}_w$ (resp. $\mathcal{ AZ }_w$) be the $\overline{k}$-vector space (resp. $k$-vector space) spanned by weight $w$ AMZVs. 
By Theorem \ref{shamz}, we derive $\mathcal{ AZ }_w\cdot\mathcal{ AZ}_{w'} \subseteq \mathcal{AZ}_{w+w'}$. We also note the $\overline{k}$-algebra $\overline{\mathcal{ AZ }}$ (resp. $k$-algebra $\mathcal{ AZ }$) generated by AMZVs.
\end{nota}

\begin{thm}\label{gonamz}
We have the following:
\begin{itemize}
\item[(i)] $\overline{\mathcal{AZ}}$ forms an weight-graded algebra, that is, $\overline{\mathcal{AZ}}=\overline{k}\bigoplus_{w\in\mathbb{N}}\overline{\mathcal{AZ}}_w$;
\item[(ii)] $\overline{\mathcal{AZ}}$ is defined over $k$, that is, we have the canonical map $\overline{k}\otimes_{k}\mathcal{AZ}\rightarrow\overline{\mathcal{AZ}}$ which is bijective.  
\end{itemize}
\end{thm}
A direct consequence of Theorem \ref{gonamz} is the following transcendence result.
\begin{cor}
Each AMZV $\zeta_A({\mathfrak a};{\boldsymbol \epsilon})$ is transcendental over $k$. 
\end{cor}

\section*{Appendix. Explicit sum-shuffle relations in a lower depth case}
We provide several steps to calculate an explicit sum-shuffle relation for $\zeta_A({\mathfrak a}; {\boldsymbol \epsilon})\zeta_A({\mathfrak b}; {\boldsymbol \lambda})$ with ${\mathfrak a}:=(a_1, a_2)\in\mathbb{N}^{2}$, ${\mathfrak b}:=(b_1)\in\mathbb{N}$, ${\boldsymbol \epsilon}:=(\epsilon_1, \epsilon_2)\in(\mathbb{F}_q^{\times})^{2}$ and ${\boldsymbol \lambda}:=(\lambda_1)\in(\mathbb{F}_q^{\times})$.
\begin{itemize} 
\item[(I)] Express $\zeta_A({\mathfrak a}; {\boldsymbol \epsilon})\zeta_A({\mathfrak b}; {\boldsymbol \lambda})$ in alternating power sums by using \eqref{extpowsum} and \eqref{chpes}:
\begin{align*}
	\zeta_A({\mathfrak a}; {\boldsymbol \epsilon})\zeta_A({\mathfrak b}; {\boldsymbol \lambda})&=\sum_{d\geq 0}S_d({\mathfrak a}; {\boldsymbol \epsilon})S_d({\mathfrak b}; {\boldsymbol \lambda})+\sum_{d>e\geq 0}S_d({\mathfrak a}; {\boldsymbol \epsilon})S_e({\mathfrak b}; {\boldsymbol \lambda})+\sum_{e>d\geq 0}S_e({\mathfrak b}; {\boldsymbol \lambda})S_d({\mathfrak a}; {\boldsymbol \epsilon})\\
&=\sum_{d\geq 0}S_d({\mathfrak a}; {\boldsymbol \epsilon})S_d({\mathfrak b}; {\boldsymbol \lambda})+\sum_{d>e\geq 0}S_d({\mathfrak a}; {\boldsymbol \epsilon})S_e({\mathfrak b}; {\boldsymbol \lambda})+\sum_{e\geq 0}S_e({\mathfrak b}, {\mathfrak a}; {\boldsymbol \lambda}, {\boldsymbol \epsilon}).
\end{align*}

\item[(II)] Express $S_d({\mathfrak a}; {\boldsymbol \epsilon})S_d({\mathfrak b}; {\boldsymbol \lambda})$ and $\sum_{d>e\geq 0}S_d({\mathfrak a}; {\boldsymbol \epsilon})S_e({\mathfrak b}; {\boldsymbol \lambda})$ as an $\mathbb{F}_p$-linear combination of $S_{d_1}(-;-)$ by using \eqref{extpowsum}, \eqref{lempowsh} and Remark \ref{hl1} (for $\Delta^{j_1}_{a_1, b_1}$, $\Delta^{j_2}_{j_1, a_2}$ and $\Delta^{j_3}_{a_2, b_1}$, see \eqref{deltadef}).

For $\sum_{d\geq 0}S_d({\mathfrak a}; {\boldsymbol \epsilon})S_d({\mathfrak b}; {\boldsymbol \lambda})$, we have the following:
\begin{align*}
&\sum_{d\geq 0}S_d({\mathfrak a}; {\boldsymbol \epsilon})S_d({\mathfrak b}; {\boldsymbol \lambda})=\sum_{d_1\geq 0}S_{d_1}({\mathfrak a}; {\boldsymbol \epsilon})S_{d_1}({\mathfrak b}; {\boldsymbol \lambda})=\sum_{d_1\geq 0}S_{d_1}(a_1;\epsilon_1)S_{<d_1}(a_2; \epsilon_2)S_{d_1}(b_1; \lambda_1)\\
&=\sum_{d_1\geq 0}S_{d_1}(a_1;\epsilon_1)S_{d_1}(b_1; \lambda_1)\sum_{d_1>d_2\geq 0}S_{d_2}(a_2; \epsilon_2)\\
&=\sum_{d_1\geq 0}\Bigl(\sum_{\substack{a_1+b_1>j_1>0\\q-1|j_1}}\Delta^{j_1}_{a_1, b_1}S_{d_1}(a_1+b_1-j_1, j_1; \epsilon_1\lambda_1, 1)\\
&{\hspace{5.2cm}}+S_{d_1}(a_1+b_1;\epsilon_1\lambda_1)\Bigr)\sum_{d_1>d_2\geq 0}S_{d_2}(a_2; \epsilon_2)\\
&=\sum_{d_1\geq 0}\biggl(\sum_{\substack{a_1+b_1>j_1>0\\q-1|j_1}}\Delta^{j_1}_{a_1, b_1}S_{d_1}(a_1+b_1-j_1; \epsilon_1\lambda_1)\sum_{d_1>d_2'\geq0}S_{d_2'}(j_1;1)\\
&{\hspace{5.2cm}}+S_{d_1}(a_1+b_1;\epsilon_1\lambda_1)\biggr)\sum_{d_1>d_2\geq 0}S_{d_2}(a_2; \epsilon_2)
\end{align*}
\begin{align*}
&=\sum_{d_1\geq 0}\biggl(\sum_{\substack{a_1+b_1>j_1>0\\q-1|j_1}}\Delta^{j_1}_{a_1, b_1}S_{d_1}(a_1+b_1-j_1; \epsilon_1\lambda_1)\biggl\{\sum_{d_1>d_2'=d_2\geq0}S_{d_2'}(j_1;1)S_{d_2}(a_2; \epsilon_2)\\
&{\hspace{1cm}}+\sum_{d_1>d_2'>d_2\geq0}S_{d_2'}(j_1;1)S_{d_2}(a_2; \epsilon_2)+\sum_{d_1>d_2>d_2'\geq0}S_{d_2}(a_2; \epsilon_2)S_{d_2'}(j_1;1)\biggr\}\\
&{\hspace{1cm}}+\sum_{d_1\geq 0}S_{d_1}(a_1+b_1, a_2;\epsilon_1\lambda_1, \epsilon_2)\biggr)\\
&=\sum_{d_1\geq 0}\biggl(\sum_{\substack{a_1+b_1>j_1>0\\q-1|j_1}}\Delta^{j_1}_{a_1, b_1}S_{d_1}(a_1+b_1-j_1; \epsilon_1\lambda_1)\\
&{\hspace{1cm}}\biggl\{\sum_{d_1>d_2'\geq0}\Bigl(\sum_{\substack{j_1+a_2>j_2>0\\q-1|j_2}}\Delta^{j_2}_{j_1, a_2}S_{d_2'}(j_1+a_2-j_2, j_2;\epsilon_2, 1)+S_{d_2'}(a_2+j_1;\epsilon_2)\Bigr)\\
&{\hspace{1cm}}+\sum_{d_1>d_2'\geq0}S_{d_2'}(j_1, a_2;1, \epsilon_2)+\sum_{d_1>d_2\geq0}S_{d_2}(a_2, j_1; \epsilon_2, 1)\biggr\}\\
&{\hspace{1cm}}+\sum_{d_1\geq 0}S_{d_1}(a_1+b_1, a_2;\epsilon_1\lambda_1, \epsilon_2)\biggr)\\
&=\sum_{\substack{a_1+b_1>j_1>0\\q-1|j_1}}\Delta^{j_1}_{a_1, b_1}\biggl(\sum_{\substack{j_1+a_2>j_2>0\\q-1|j_2}}\Delta^{j_2}_{j_1, a_2}\sum_{d_1\geq 0}S_{d_1}(a_1+b_1-j_1, j_1+a_2-j_2, j_2; \epsilon_1\lambda_1, \epsilon_2, 1)\\
&{\hspace{1cm}}+\sum_{d_1\geq 0}S_{d_1}(a_1+b_1-j_1, j_1+a_2;\epsilon_1\lambda_1, \epsilon_2)+\sum_{d_1\geq0}S_{d_1}(a_1+b_1-j_1, j_1, a_2;\epsilon_1\lambda_1, 1, \epsilon_2)\\
&{\hspace{1cm}}+\sum_{d_1\geq0}S_{d_1}(a_1+b_1-j_1, a_2, j_1;\epsilon_1\lambda_1, \epsilon_2, 1)\biggr)+\sum_{d_1\geq 0}S_{d_1}(a_1+b_1, a_2;\epsilon_1\lambda_1, \epsilon_2).
\end{align*}
For $\sum_{d>e\geq 0}S_d({\mathfrak a}; {\boldsymbol \epsilon})S_e({\mathfrak b}; {\boldsymbol \lambda})$, we have the following:
\begin{align*}
&\sum_{d>e\geq 0}S_d({\mathfrak a}; {\boldsymbol \epsilon})S_e({\mathfrak b}; {\boldsymbol \lambda})=\sum_{d_1>e_1\geq 0}S_{d_1}({\mathfrak a}; {\boldsymbol \epsilon})S_{e_1}({\mathfrak b}; {\boldsymbol \lambda})=\sum_{d_1>e_1\geq 0}S_{d_1}(a_1;\epsilon_1)S_{<d_1}(a_2; \epsilon_2)S_{e_1}(b_1; \lambda_1)\\
&=\sum_{d_1\geq 0}S_{d_1}(a_1;\epsilon_1)\sum_{d_1>d_2\geq 0}S_{d_2}(a_2; \epsilon_2)\sum_{d_1>e_1\geq 0}S_{e_1}(b_1; \lambda_1)\\
&=\sum_{d_1\geq 0}S_{d_1}(a_1;\epsilon_1)\biggl(\sum_{d_1>d_2=e_1\geq 0}S_{d_2}(a_2; \epsilon_2)S_{e_1}(b_1; \lambda_1)+\sum_{d_1>d_2>e_1\geq 0}S_{d_2}(a_2; \epsilon_2)S_{e_1}(b_1; \lambda_1)\\
&{\hspace{7.2cm} +\sum_{d_1>e_1>d_2\geq 0}S_{e_1}(b_1; \lambda_1)S_{d_2}(a_2; \epsilon_2)\biggr)}\\
&=\sum_{d_1\geq 0}S_{d_1}(a_1;\epsilon_1)\biggl(\sum_{d_1>d_2=e_1\geq 0}\biggl\{\sum_{\substack{a_2+b_1>j_3>0\\q-1|j_3}}\Delta^{j_3}_{a_2, b_1}S_{d_2}(a_2+b_1-j_3, j_3; \epsilon_2\lambda_1, 1)\\
&{\hspace{1cm}+S_{d_2}(a_2+b_1;\epsilon_2\lambda_1)\biggr\}+\sum_{d_1>d_2>\geq 0}S_{d_2}(a_2, b_1; \epsilon_2, \lambda_1)+\sum_{d_1>e_1>\geq 0}S_{e_1}(b_1, a_2; \lambda_1, \epsilon_2)\biggr)}
\end{align*}
\begin{align*}
&=\sum_{\substack{a_2+b_1>j_3>0\\q-1|j_3}}\Delta^{j_3}_{a_2, b_1}\sum_{d_1\geq 0}S_{d_1}(a_1, a_2+b_1-j_3, j_3;\epsilon_1, \epsilon_2\lambda_1, 1)+\sum_{d_1\geq 0}S_{d_1}(a_1, a_2+b_1;\epsilon_1, \epsilon_2\lambda_1)\\
&{\hspace{3.2cm} +\sum_{d_1\geq 0}S_{d_1}(a_1, a_2, b_1;\epsilon_1, \epsilon_2, \lambda_1)+\sum_{d_1\geq 0}S_{d_1}(a_1, b_1, a_2;\epsilon_1, \lambda_1, \epsilon_2)}.
\end{align*}
\item[(III)] Use step (II) and \eqref{chpes}. Then we obtain a sum-shuffle relation for $\zeta_A({\mathfrak a}; {\boldsymbol \epsilon})\zeta_A({\mathfrak b}; {\boldsymbol \lambda})$ with explicit coefficients.
 \begin{align*}
	&\zeta_A({\mathfrak a}; {\boldsymbol \epsilon})\zeta_A({\mathfrak b}; {\boldsymbol \lambda})=\sum_{d\geq 0}S_d({\mathfrak a}; {\boldsymbol \epsilon})S_d({\mathfrak b}; {\boldsymbol \lambda})+\sum_{d>e\geq 0}S_d({\mathfrak a}; {\boldsymbol \epsilon})S_e({\mathfrak b}; {\boldsymbol \lambda})+\sum_{e\geq 0}S_e({\mathfrak b}, {\mathfrak a}; {\boldsymbol \lambda}, {\boldsymbol \epsilon})
\end{align*}
\begin{align*}
&=\sum_{\substack{a_1+b_1>j_1>0\\q-1|j_1}}\Delta^{j_1}_{a_1, b_1}\biggl(\sum_{\substack{j_1+a_2>j_2>0\\q-1|j_2}}\Delta^{j_2}_{j_1, a_2}\zeta_A(a_1+b_1-j_1, j_1+a_2-j_2, j_2; \epsilon_1\lambda_1, \epsilon_2, 1)\\
&{\hspace{1cm}}+\zeta_A(a_1+b_1-j_1, j_1+a_2;\epsilon_1\lambda_1, \epsilon_2)+\zeta_A(a_1+b_1-j_1, j_1, a_2;\epsilon_1\lambda_1, 1, \epsilon_2)\\
&{\hspace{1cm}}+\zeta_A(a_1+b_1-j_1, a_2, j_1;\epsilon_1\lambda_1, \epsilon_2, 1)\biggr)+\zeta_A(a_1+b_1, a_2;\epsilon_1\lambda_1, \epsilon_2)\\
&\quad+\sum_{\substack{a_2+b_1>j_3>0\\q-1|j_3}}\Delta^{j_3}_{a_2, b_1}\zeta_A(a_1, a_2+b_1-j_3, j_3;\epsilon_1, \epsilon_2\lambda_1, 1)+\zeta_A(a_1, a_2+b_1;\epsilon_1, \epsilon_2\lambda_1)\\
&{\hspace{1cm}+\zeta_A(a_1, a_2, b_1;\epsilon_1, \epsilon_2, \lambda_1)+\zeta_A(a_1, b_1, a_2;\epsilon_1, \lambda_1, \epsilon_2)}\\
&\quad+\zeta_A({\mathfrak b}, {\mathfrak a}; {\boldsymbol \lambda}, {\boldsymbol \epsilon}).
\end{align*}
\end{itemize}

\section*{Acknowledgments}
The author is deeply grateful to Professor H. Furusho for giving him many fruitful comments on this paper and guiding him to this area. He also gratefully acknowledges Professor C.-Y. Chang for indicating the author towards this project by suggesting  definition of AMZVs and giving some of observations to him. This paper could not have been written without their continuous encouragements. He further would like to thank JSPS Overseas Challenge Program for Young Researchers and National Center for Theoretical Sciences in Hsinchu for their support during his stay at National Tsing Hua University. This work was also supported by JSPS KAKENHI Grant Number JP18J15278.

\end{document}